\documentclass{amsart}

\usepackage{amsmath, amssymb, amsthm, amscd, amsfonts, dsfont}
\usepackage{bbold}
\usepackage{hyperref}
\usepackage{macros}

\numberwithin{equation}{section}

\standardrenews

\ignore{
\newcommand{\ignore}[1]{}
\DeclareMathOperator{\id}{id}

\newcommand{\N}{\mathbb N}
\renewcommand{\P}{\mathbb P}

\newcommand{\ww}{\mathbf w}
\newcommand{\xx}{\mathbf x}

\renewcommand{\AA}{\mathcal A}
\newcommand{\LL}{\mathcal L}
\newcommand{\PP}{\mathcal P}

\newcommand{\MM}{\mathcal M}

\newcommand{\bfbeta}{{\boldsymbol\beta}}
\newcommand{\bfgamma}{{\boldsymbol\gamma}}
\newcommand{\FF}{\mathcal F}
\newcommand{\TT}{\mathcal T}
\newcommand{\butnot}{\setminus}

\renewcommand{\text}{\textup}
\newcommand{\given}{\upharpoonleft}
\newcommand{\w}{\widetilde}

\newif\ifdraft\drafttrue

\newcommand{\authordavid}{\ifdraft{\author{DS}}\else{
\author{David Simmons}
\address{434 Hanover Ln, Irving, TX 75062, USA}
\email{david9550\at gmail.com}
\urladdr{\url{https://sites.google.com/view/davidsimmonsmath2/}}
}\fi}

\newtheorem{theorem}{Theorem}[section]
\numberwithin{theorem}{section}
\newtheorem{lemma}[theorem]{Lemma}
\newtheorem{proposition}[theorem]{Proposition}
\newtheorem{corollary}[theorem]{Corollary}

\theoremstyle{definition}

\newtheorem{definition}[theorem]{Definition}
\newtheorem{observation}[theorem]{Observation}
\newtheorem{example}[theorem]{Example}
\newtheorem{remark}[theorem]{Remark}

}% end ignore

\newcommand{\LQ}{\text{``}}
\newcommand{\RQ}{\text{''}}

\DeclareMathOperator{\eval}{eval}
\newcommand{\NMID}{\mathbf{NMID}}
\newcommand{\ID}{\mathbf{ID}}
\newcommand{\GDST}{\mathbf{GDST}}
\newcommand{\ST}{\mathbf{ST}}
\newcommand{\KP}{\mathbf{KP}}

\draftfalse

\begin{document}

\title{Gradualist descriptionalist set theory}

\begin{abstract}
We introduce a formal language $\GDST$ (gradualist descriptionalist set theory)  with a family of interpretations indexed by ordinals, as well as a sublanguage $\NMID$ (the language of not necessarily monotonic inductive definitions), and show that the assertion that all propositions in $\NMID$ have well-defined truth values is equivalent to the existence for each $k\in\N$ of a sequence of ordinals $\eta_0 < \ldots < \eta_k$ such that for each $i < k$, $\eta_i$ is \emph{$\eta_{i+1}$-reflecting}, a notion we introduce which implies being $\Pi_n$-reflecting for all $n\in\N$ (and in particular being admissible and recursively Mahlo).
\end{abstract}
\authordavid

\maketitle

\section{Introduction}

%[Cf. $\ID$ vs $\NMID$ and vs the version of $\NMID$ where you don't require $T(w) \supseteq w$\ddr]

The following questions are closely linked:
\begin{itemize}
\item What mathematical objects exist?
\item What mathematical questions have definite answers?
\end{itemize}
In \cite[Section 4]{Simmons11}, the author proposed a possible answer to these questions in the form of an interpreted formal language $\GDST$, ``gradualist descriptionalist set theory'', the motivating intuitions of which are that
\begin{itemize}
\item Gradualism: The universe of sets cannot be viewed as a totality at any given point in time; rather, it is always growing. Cf. \cite[p.503]{Maddy1} and the references therein.
\item Descriptionalism: Every mathematical object should be describable in principle, where by ``describable'' we mean something like ``assigned a natural number (or equivalently a string in a countable language) that uniquely identifies it''. Descriptionalism is somewhat similar to predicativism (cf. e.g. \cite{Weaver_predicativism} for an overview), and also to Nik Weaver's ``constructive countablism'' \cite{Weaver_constructive_countablism}.
%and predicativism \cite{BFPS}, with the difference that descriptionalism takes the sequence of ordinals at face value as long as they can be uniquely assigned natural numbers. [Descriptionalism vs predicativism: descriptionalism takes the ordinals at face value as long as they can be described, whereas predicativism only takes the natural numbers at face value / predicativism seems an inconsistent framework given that it is bounded by a specific ordinal $\Gamma_0$; those who take predicativism beyond $\Gamma_0$ still do not take ordinals at face value -- another way of putting it is I am inspired by predicativism, but don't want to claim to be doing the same thing necessarily\internal]
\end{itemize}
See Section \ref{sectiondef} for a precise definition of $\GDST$. The semantics of $\GDST$ can be taken in two ways: either ``at face value'', referring to a gradualist descriptionalist view of the constructible universe $L$, or ``embedded in a metatheory''. In this note we will utilize both points of view, but focus on the second, using Kripke--Platek set theory $\KP$ as our metatheory.

In Section \ref{sectiondef} we introduce the language $\GDST$. In Section \ref{sectionaxiomatization} we introduce an axiomatization of $\GDST$ and compare it to the axiomatization $\KP$ of $\ST$. In Section \ref{sectioncanonical} we deal with the question of whether there is a canonical ordinal at which to halt the constructive hierarchy in the context of $\GDST$. In Section \ref{sectionreflecting} we introduce the notion of reflecting ordinals and compare them with other large ordinal properties; in particular, we show that reflecting ordinals are admissible. In Section \ref{sectionNMID} we introduce a sub-language of $\GDST$, the language of not necessarily monotonic inductive definitions $\NMID$, and show that the existence of $k$-fold reflecting ordinals for all $k\in\N$ implies that propositions in the language $\NMID$ have well-defined truth values. In Section \ref{sectionreverse} we prove the converse of this assertion, thus justifying the existence of $k$-fold reflecting ordinals as being in a sense predicative.

\section{Definition of $\GDST$}
\label{sectiondef}

Throughout this paper we use Kripke--Platek set theory $\KP$ with $\Delta_0$ induction (cf. Definition \ref{definitionKP} and Remark \ref{remarkKP}) as our metatheory. We use the notation $(L_\alpha)$ for the usual transfinite sequence of stages of the constructible universe $L$.

\begin{definition}
The language $\GDST$ is the two-sorted (sets and ordinals) language whose syntax consists of the following formulas and terms:
\begin{itemize}
\item[(1)] Atomic formulas are of the form $x\in y$, $x=y$, and $x\downarrow$ where $x$ and $y$ are either set or ordinal terms (mixing allowed). Here $x\downarrow$ means ``$x$ exists'' or ``the term $x$ has a referent''. When $\alpha$ and $\beta$ are both ordinal terms we write $\alpha < \beta$ interchangeably with $\alpha \in \beta$.
\item[(2)] Formulas are closed under the logical operations ``and'', ``or'', and ``not'' (but not ``implies'', although we use the shorthand $P\rightarrow Q$ for $\neg P \vee Q$).
\item[(3)] If $P(x)$ is a formula depending on (possibly among others) a set variable $x$ and if $X$ is a set term, then $\forall x\in X \; P(x)$ and $\exists x\in X \; P(x)$ are formulas. Similarly, if $P(\alpha)$ is a formula depending on an ordinal variable $\alpha$ and $A$ is an ordinal term, then $\forall \alpha < A \; P(\alpha)$ and $\exists \alpha < A \; P(\alpha)$ are formulas.
\item[(4)] If $P(x)$ is a formula depending on a set variable $x$ and $X$ is a set term, then $\{x\in X : P(x)\}$ is a set term.
\item[(5)] If $A$ is an ordinal term then $L_A$ is a set term.
\item[(6)] If $P(\alpha)$ is a formula depending on an ordinal variable $\alpha$, then $\min\{\alpha : P(\alpha)\}$ is an ordinal term.
\end{itemize}
The semantics of $\GDST$ are defined as follows: for any given ordinal $\eta$ and for each $M,N\in \N$, we define maps
\begin{align*}
\iota_\eta = \iota_{\eta,M,N} &: \FF_\GDST(M,N) \times L_\eta^M \times \eta^N \to \{\top,\bot,U\}\\
j_\eta = j_{\eta,M,N} &: \TT_\GDST(M,N) \times L_\eta^M \times \eta^N \to \eta\cup \{U\}\\
j'_\eta = j_{\eta,M,N}' &: \TT_\GDST'(M,N) \times L_\eta^M \times \eta^N \to L_\eta\cup \{U\}
\end{align*}
where $\FF_\GDST(M,N)$, $\TT_\GDST(M,N)$, and $\TT_\GDST'(M,N)$ denote the sets of formulas, ordinal terms, and set terms of $\GDST$ with $M$ free set variables and $N$ free ordinal variables, respectively. Here $U$ is meant to denote ``undefined''. The maps $\iota_\eta$, $j_\eta$, and $j'_\eta$ are defined recursively in the obvious way, noting that
\begin{itemize}
\item $\iota_\eta(X\downarrow)$ is undefined ($U$) if $j'_\eta(X) = U$, and is true ($\top$) otherwise.
\item $\iota_\eta(\forall x\in X \; P(x))$ is undefined if $j_\eta(A) = U$, otherwise false ($\bot$) if there exists $x\in j'_\eta(X)$ such that $\iota_\eta(P,x) = \bot$, otherwise undefined if there exists $x\in j'_\eta(X)$ such that $\iota_\eta(P,x) = U$, otherwise true ($\top$). $\iota_\eta(\exists x\in X \; P(x))$ can be described dually by de Morgan's laws.
\item $j'_\eta(\{x\in X : P(x)\}) = U$ if $\iota_\eta(P,x) = U$ for any $x\in j'_\eta(X)$, otherwise $j'_\eta(\{x\in X : P(x)\}) = \{x\in j'_\eta(X) : \iota_\eta(P,x)\}$.
\item $j'_\eta(L_A) = L_{j_\eta(A)}$, where $L_\alpha$ is the $\alpha$th stage of the constructible universe $L$.
\item $j_\eta(\min\{\alpha : P(\alpha)\})$ is equal to $\alpha < \eta$ if and only if both (A) $\iota_\eta(P,\alpha) = \top$ and (B) $\iota_\eta(P,\beta) = \bot$ for all $\beta < \alpha$. If this is not true for any $\alpha < \eta$, then $j_\eta(\min\{\alpha : P(\alpha)\}) = U$. Note that this means that $\iota_\eta(P,\beta)$ may be undefined for some $\beta \in (\alpha,\eta)$ without leading $j_\eta(\min\{\alpha : P(\alpha)\})$ to be undefined.
%\item $j'_\eta(x \in L_\alpha : P(x)) = x$ if $j_\eta(j'_\eta(\{x\in L_\cdot : P(x)\},\alpha)) = \{x\}$, $j'_\eta(x \in L_\alpha : P(x)) = U$ otherwise.
\end{itemize}
The triple $I_\eta = (\iota_\eta,j_\eta,j'_\eta)$ is said to be the \emph{interpretation} of $\GDST$ with respect to the model $L_\eta$.
\end{definition}

\begin{remark}
We will often use shorthand to indicate formulas and terms of $\GDST$ constructed using obvious techniques, for example $A + 1 := \min\{\alpha : A < \alpha\}$.
\end{remark}

In what follows we use $\xx$ to denote an element of $L_\eta^M \times \eta^N$, and use the shorthand $\iota_\eta(P[\xx]) := \iota_\eta(P,\xx)$, $j_\eta(A[\xx]) := j_\eta(A,\xx)$, $j'_\eta(X[\xx]) := j'_\eta(X,\xx)$.

\begin{remark}
\label{remarkST}
Formulas $P(\xx)$ built using only constructors (1)-(3) are just the $\Delta_0$ formulas of standard set theory $\ST$ (the one-sorted first-order predicate logic with binary predicates $\in$ and $=$). However, all such formulas have at least one free variable, in contrast to $\GDST$ which has formulas and terms with no free variables (built using constructor (6)). In the sequel we will denote the set of $\Delta_0$ formulas of $\ST$ by $\Delta_0(\ST)$.
\end{remark}

\begin{remark}
\label{remarkQB}
If $P$ is a formula of $\ST$, then by replacing all unrestricted quantifiers with bounded quantifiers of the form $x\in \MM$ we get a $\Delta_0$ formula $P'(\MM)$ such that $P$ holds of a model $\MM$ if and only if $P'(\MM)$ holds.
\end{remark}

\begin{observation}[Monotonicity]
\label{observationmonotone}
If $\eta < \tau$ and $\iota_\eta(P[\xx]) \neq U$ (resp. $j_\eta(A[\xx]) \neq U$, $j'_\eta(X[\xx]) \neq U$), then $\iota_\tau(P[\xx]) = \iota_\eta(P[\xx])$ (resp. $j_\tau(A[\xx]) = j_\eta(A[\xx])$, $j'_\tau(X[\xx]) = j'_\eta(X[\xx])$).
\end{observation}
\begin{observation}[Continuity]
\label{observationcontinuous}
If $\tau$ is a limit ordinal then for all $P,A,X,\xx$, we have $\iota_\tau(P[\xx]) = \iota_\eta(P[\xx])$, $j_\tau(A[\xx]) = j_\eta(A[\xx])$, and $j'_\tau(X[\xx]) = j'_\eta(X[\xx])$ for all $\eta < \tau$ sufficiently close to $\tau$.
\end{observation}
\begin{observation}
\label{observationignoreconstructor}
If $P$ is a formula not containing any instances of construction rule (6), $\eta < \tau$, and $\xx\in L_\eta^M \times \eta^N$, then $\iota_\eta(P[\xx]) = \iota_\tau(P[\xx]) \in \{\top,\bot\}$, in which case we can just write $\iota(P[\xx])$.
\end{observation}

The proofs are easy inductions on the depths of $P$, $A$, and $X$.

\begin{proposition}
\label{propositionselfmeta}
There exist set terms $\eval_\iota$, $\eval_j$, and $\eval_{j'}$ such that for all limit $\eta,\tau$ with $\eta < \tau$, for all $M,N\in\N$, for all $P\in \FF_\GDST(M,N)$, and for all $\xx\in L_\eta^M \times \eta^N$ we have
\[
j'_\tau(\eval_\iota[P,\xx,\eta])  = \begin{cases}
\iota_\eta(P[\xx]) & \text{ if } \iota_\eta(P[\xx]) \neq U\\
U' & \text{ if } \iota_\eta(P[\xx]) = U
\end{cases}
\]
and similarly for $\eval_j$, $\eval_{j'}$. Here $U'$ is a marked element distinct from $U$, so that $j'_\tau(\eval_\iota[P,\xx,\eta]) \neq U$ for all $P,\xx,\eta$.
\end{proposition}
Essentially, this proposition states that for $\eta < \tau$, $L_\eta$ is a model of $\GDST$ whose semantics can be defined relative to the semantics of the model $L_\tau$ of $\GDST$. The proof is straightforward and we omit it.

\begin{corollary}
\label{corollaryselfmeta}
There exist a formula $\overline\eval_\iota$, an ordinal term $\overline\eval_j$, and a set term $\overline\eval_{j'}$ such that for all limit $\eta$, for all $M,N\in\N$, for all $P\in \FF_\GDST(M,N)$, and for all $\xx\in L_\eta^M \times \eta^N$ we have
\[
\iota_\eta(\overline\eval_\iota[P,\xx]) = \iota_\eta(P[\xx])
\]
and similarly for $\overline\eval_j$, $\overline\eval_{j'}$.
\end{corollary}
\begin{proof}
Let $\overline\eval_\iota(P[\xx]) = \eval_\iota(P,\xx,\min\{\eta : \eval_\iota(P,\xx,\eta) \neq U'\})$, and similarly for $\overline\eval_j$, $\overline\eval_{j'}$.
\end{proof}

Note that this corollary means that the language $\GDST$ is ``selfmeta'', or capable of expressing its own truth predicate. We recall that Tarski's theorem on the indefinability of truth \cite{Tarski} implies that it is impossible for a language to express its own truth predicate if it satisfies certain technical conditions, one of which is that it has an axiomatization with respect to which statements of the form $P \rightarrow Q$ are allowed to be proven using natural deduction. (Note that both classical and intuitionistic logic satisfy this condition.) However, $\GDST$ does not satisfy this condition, since the axiomatization $\AA_\GDST$ we introduce in Section \ref{sectionaxiomatization} does not allow implication statements to be proven by natural deduction, as is in fact entailed by the semantics of $\GDST$ as described above.

%for two reasons: first of all, $\GDST$ does not have a notion of implication, see \cite[Appendix D]{Simmons11} for more on the motivation for this decision. Secondly, we have not introduced an axiomatization for $\GDST$ but rather we prove statements about it using the metatheory $\KP$, so there is no notion of a natural deduction proof that would yield a statement of the form $P \Rightarrow Q$. It would be possible to introduce an axiomatization of $\GDST$, but it is not relevant for our purposes here.

%Also, if $M'$ is a transitive submodel of $M$ with $\eta\in M'$, then $M' \satisfies_\eta P$ implies $M \satisfies_\eta P$. (This is probably too technical)

\section{An axiomatization of $\GDST$}
\label{sectionaxiomatization}

\begin{definition}
Let $P_1,\ldots,P_k \;\Rightarrow\, Q$ denote a statement of the form ``$P_1,\ldots,P_k$ implies $Q$''. We define the \emph{standard axiomatization} of $\GDST$, $\AA_\GDST$, to be the set of implications generated by the following inference rules:
\begin{itemize}
\item[(0)] (Cf. \cite[Appendix B]{Simmons11}) For all $i = 1,\ldots,k$, $P_1,\ldots,P_k \;\Rightarrow\; P_i$ is a theorem of $\AA_\GDST$ (reflexivity). If $P_1,\ldots,P_k \;\Rightarrow\; Q$ and $P_1,\ldots,P_k,Q \;\Rightarrow\; R$ are theorems of $\AA_\GDST$, then so is $P_1,\ldots,P_k \;\Rightarrow\; R$ (modus ponens). If $P_1,\ldots,P_k \;\Rightarrow\; Q(x)$ and $P_1,\ldots,P_k \;\Rightarrow\; X\downarrow$ are theorems of $\AA_\GDST$ and $x$ does not appear free in $P_1,\ldots,P_k$, then $P_1,\ldots,P_k \;\Rightarrow\; \forall x\in X \; Q(x)$ is a theorem of $\AA_\GDST$ (generalization). If $P_1,\ldots,P_k,Q_1 \;\Rightarrow\; R$, $P_1,\ldots,P_k,Q_2 \;\Rightarrow\; R$, and $P_1,\ldots,P_k \;\Rightarrow\; Q_1\vee Q_2$ are theorems of $\AA_\GDST$, then so is $P_1,\ldots,P_k \;\Rightarrow\; R$ ($\vee$-elimination).
\item[(1)] Writing $P\downarrow$ as shorthand for $P \vee \neg P$, we include for each constructor (1)-(5) an implication: if the components of an expression are well-defined, then the expression is well-defined. For example, if $P_1,\ldots,P_k \;\Rightarrow\; A\downarrow$ is a theorem of $\AA_\GDST$ (resp. $P_1,\ldots,P_k \;\Rightarrow\; X\downarrow$ and $P_1,\ldots,P_k, x\in X \;\Rightarrow\; P(x)\downarrow$ are both theorems of $\AA_\GDST$), then so is $P_1,\ldots,P_k \;\Rightarrow\, L_A\downarrow$ (resp. $P_1,\ldots,P_k \Rightarrow \{x\in X : P(x)\}\downarrow$).
\item[(2)] If $P\downarrow$ and $Q\downarrow$, then we use the shorthand $P \rightarrow Q$ for $\neg P \vee Q$. We include all standard axioms of (classical, Hilbert) propositional logic conditional on their components being well-defined, such as $P\downarrow, Q\downarrow \;\Rightarrow\, P \rightarrow (Q \rightarrow P)$, together with the implication form of modus ponens ($P, (P\rightarrow Q) \;\Rightarrow\, Q$), for which we do not require $P\downarrow$ or $Q\downarrow$ to hold. Note that we do not include the standard axioms of predicate logic unless $P\downarrow$ holds, since e.g. $P \;\rightarrow\, P$ is shorthand for $\neg P \vee P$ which is the law of excluded middle.
\item[(3)] If the relevant formulas and terms are well-defined, we include all standard axioms and inference rules of predicate logic, e.g. $X\downarrow , Y\downarrow , \forall x\in X \; P(x)\downarrow \;\Rightarrow\, (\forall x\in X \; P(x) \wedge Y\in X) \rightarrow P(Y)$ and $P(x) , X\downarrow \;\Rightarrow\, P(X)$.
\item[(4)]
\[
X\downarrow , Y\downarrow , \forall x\in X \; P(x)\downarrow \;\Rightarrow\, (Y\in \{x\in X : P(x)\} \leftrightarrow (Y\in X \wedge P(Y)).
\]
\item[(5)] Definition of $L_\alpha$:
\begin{align*}
\forall x\in L_\alpha \; &\exists \beta < \alpha \; \exists P\in \Delta_0(\ST) \; \exists M\in\N \; \exists \zz\in L_\beta^M \;\; x = \{y \in L_\beta : \eval_\iota(P,(y,\zz),\beta) = \top\}\\
&\forall \beta < \alpha \; \forall P\in \Delta_0(\ST) \; \forall M\in\N \; \forall \zz\in L_\beta^M \;\; \{y \in L_\beta : \eval_\iota(P,(y,\zz),\beta) = \top\} \in L_\alpha.
\end{align*}
\item[(6)] Letting $A_P = \min\{\alpha : P(\alpha)\}$, we have
\[
P(A), \forall \beta < A \; P(\beta)\downarrow \;\Rightarrow\, A_P\downarrow \;\Rightarrow\, (P(A_P) \wedge \forall \beta < A_P \; \neg P(\beta)).
\]
(This axiom also serves as our induction axiom.)
\item[(7)] Transitivity of ordinals:
\[
\alpha < \beta < \gamma \rightarrow \alpha < \gamma.
\]
\item[(8)] Extensionality:
\[
\forall z\in x \; z\in y \wedge \forall z\in y \; z\in x \;\rightarrow\; x = y.
\]
\end{itemize}
\end{definition}

\ignore{
% Old version
\begin{definition}
Let $P_1,\ldots,P_k \;\Rightarrow\, Q$ denote an inference rule of the form ``given $P_1,\ldots,P_k$, infer $Q$''. We define the \emph{standard axiomatization and inference rules} of $\GDST$, $\AA_\GDST$, as follows:
\begin{itemize}
\item[(1)] Writing $P\downarrow$ as shorthand for $P \vee \neg P$, we include for each constructor (1)-(5) an inference rule: if the components of an expression are well-defined, then the expression is well-defined. For example, $A\downarrow \;\Rightarrow\, L_A\downarrow$ and $X\downarrow , P(x)\downarrow \;\Rightarrow\, \{x\in X : P(x)\} \downarrow$.
\item[(2)] If $P\downarrow$ and $Q\downarrow$, then we use the shorthand $P \rightarrow Q$ for $\neg P \vee Q$. We include all standard axioms of (classical, Hilbert) propositional logic conditional on their components being well-defined, such as $P\downarrow, Q\downarrow \;\Rightarrow\, P \rightarrow (Q \rightarrow P)$, together with the inference rule of modus ponens ($P, (P\rightarrow Q) \;\Rightarrow\, Q$), for which we do not require $P\downarrow$ or $Q\downarrow$ to hold. Note that we do not include the standard axioms of predicate logic unless $P\downarrow$ holds, since e.g. $P \;\rightarrow\, P$ is shorthand for $\neg P \vee P$ which is the law of excluded middle.
\item[(3)] If the relevant formulas and terms are well-defined, we include all standard axioms and inference rules of predicate logic, e.g. $X\downarrow , Y\downarrow , \forall x\in X \; P(x)\downarrow \;\Rightarrow\, (\forall x\in X \; P(x) \wedge Y\in X) \rightarrow P(Y)$ and $P(x) , X\downarrow \;\Rightarrow\, P(X)$.
\item[(4)]
\[
X\downarrow , Y\downarrow , P(x)\downarrow \;\Rightarrow\, (Y\in \{x\in X : P(x)\} \leftrightarrow (Y\in X \wedge P(Y)).
\]
\item[(5)] Definition of $L_\alpha$:
\begin{align*}
\forall x\in L_\alpha \; &\exists \beta < \alpha \; \exists P\in \Delta_0(\ST) \; \exists M\in\N \; \exists \zz\in L_\beta^M \;\; x = \{y \in L_\beta : \eval_\iota(P,(y,\zz),\beta) = \top\}\\
&\forall \beta < \alpha \; \forall P\in \Delta_0(\ST) \; \forall M\in\N \; \forall \zz\in L_\beta^M \;\; \{y \in L_\beta : \eval_\iota(P,(y,\zz),\beta) = \top\} \in L_\alpha.
\end{align*}
\item[(6)] Letting $A_P = \min\{\alpha : P(\alpha)\}$, we have
\[
P(A), \forall \beta < A \; P(\beta)\downarrow \;\Rightarrow\, A_P\downarrow \;\Rightarrow\, (P(A_P) \wedge \forall \beta < A_P \; \neg P(\beta)).
\]
\item[(7)] Transitivity of ordinals:
\[
\alpha < \beta < \gamma \rightarrow \alpha < \gamma.
\]
\item[(8)] Extensionality:
\[
\forall z\in x \; z\in y \wedge \forall z\in y \; z\in x \;\rightarrow\; x = y.
\]
\end{itemize}
\end{definition}
}% end ignore

It is easy to check (e.g. in $\KP$ or $\AA_\GDST$) that this axiomatization is truth-preserving, that is, the set
\[
\{\LQ P_1(\xx),\ldots,P_k(\xx) \;\Rightarrow\; Q(\xx)\RQ : \forall \xx\in L_\eta^M \times \eta^N \; \iota_\eta(P_1[\xx]) = \cdots = \iota_\eta(P_k[\xx]) = \top \;\Rightarrow\; \iota_\eta(Q[\xx]) = \top\}
\]
is preserved by the inference rules of $\AA_\GDST$.
%if $P_1(\xx),\ldots,P_k(\xx),Q(\xx)$ are formulas such that $\iota_\eta(P_i[\xx]) = \top$ for all $\xx \in L_\eta^M \times \eta^N$ and $i = 1,\ldots,k$, and $Q(\xx)$ is a formula produced according to the inference rules of $\AA_\GDST$, then $\iota_\eta(Q[\xx]) = \top$ for all $\xx \in L_\eta^M \times \eta^N$.

Note that this axiomatization says nothing about the value of $\eta$ with respect to which the language $\GDST$ is interpreted. Indeed, any value of $\eta$ is consistent with $\AA_\GDST$. In what follows we will be interested in extensions $\AA_\GDST + A\downarrow$, where $A$ is an ordinal term that will be assumed to exist. Such an extension is consistent with any value of $\eta$ such that $j_\eta(A) \neq U$, in particular by Observation \ref{observationmonotone} if it is consistent with some $\eta$ then it is also consistent with all $\tau > \eta$.

Thus, working within the context of $\KP$, we see that for every ordinal $\eta$ we have a model $(L_\eta,\eta)$ of $\AA_\GDST$. Next we want to do the reverse: working within the axiomatization $\AA_\GDST$ we want to prove that for certain ordinals $\eta$, $L_\eta$ is a model of $\KP$. Such ordinals are called \emph{admissible}:

\begin{definition}
\label{definitionKP}
Recall that an ordinal $\eta$ is called \emph{admissible} if $L_\eta$ is a model of Kripke--Platek set theory $\KP$. We recall the axioms of $\KP$ as follows. In all cases when we refer to an arbitrary formula $P$ we allow it to have free variables beyond those listed, which can range over arbitrary sets.
\begin{itemize}
\item (Extensionality) If $x\subseteq y$ and $y\subseteq x$, then $x = y$.
\item (Set induction) If $P(x)$ is a formula of $\ST$ such that for all $x$, if $P(y)$ holds for all $y\in x$, then $P(x)$ holds, then $P(x)$ holds for all $x$.
\item (Empty set) There exists $\emptyset$ such that $x\notin\emptyset$ for all $x$.
\item (Pairing) For all $x,y$ there exists $\{x,y\}$ such that for all $z$, we have $z\in \{x,y\}$ if and only if $z = x$ or $z = y$.
\item (Union) For all $x$ there exists $\bigcup(x)$ such that for all $z$, we have $z\in \bigcup(x)$ if and only if there exists $y\in x$ such that $z\in y$.
\item ($\Delta_0$-separation) If $P(x)$ is a $\Delta_0$ formula of $\ST$, then for all $x$ there exists a set $\{y\in x : P(y)\}$ such that for all $z$, we have $z\in \{y\in x : P(y)\}$ if and only if both $z\in x$ and $P(z)$ holds.
\item ($\Delta_0$-collection) Given any $\Delta_0$ formula $P(x,y)$ of $\ST$ and any set $z$, if for every $x\in z$ there exists $y$ such that $P(x,y)$ holds, then there exists a set $w$ such that for every $x\in z$ there exists $y\in w$ such that $P(x,y)$ holds.
\item (Infinity) The set of natural numbers $\N$ exists.
\end{itemize}
\end{definition}

\begin{remark}
\label{remarkKP}
It has been suggested by Weaver \cite[\sectionsymbol 11]{Weaver_well_ordering} that the system $\KP$ is impredicative due to the fact that the formula $P$ in the set induction schema is allowed to be any formula of $\ST$, rather than being required to be a $\Delta_0$ formula. Actually, a $\Delta_0$ version of the set induction schema suffices for our purposes in the present paper (except for the proof of Proposition \ref{propositionalphathetaw}), but we include the full schema in Definition \ref{definitionKP} to be compatible with the literature. The $\Delta_0$ version of the set induction schema implies that $L_\alpha$ satisfies the full induction schema for any ordinal $\alpha$, so changing which set induction schema we use does not change which ordinals are admissible.
\end{remark}

\begin{proposition}[$\AA_\GDST$]%\footnote{Note that to interpret a natural language statement as a theorem of $\AA_\GDST$, one must first write it in the form $A \Rightarrow B$.}]
If $\eta$ is a limit ordinal greater than $\omega$, then all axioms of $\KP$ except for possibly the axiom of $\Delta_0$-collection are satisfied by $L_\eta$.
\end{proposition}
\begin{proof}
We show that set induction and $\Delta_0$-separation are satisfied, leaving the others as exercises. Let $P(x)$ be a formula of $\ST$ such that for all $x$, if $P(y)$ holds for all $y\in x$, then $P(x)$ holds, and let $P'(x,\MM)$ be as in Remark \ref{remarkQB}. Then $P'(x,L_\eta)\in \Delta_0(\ST)$, so $P'(x,L_\eta) \downarrow$ by axiom (1). Define $Q(\alpha) \equiv \forall x\in L_\alpha \; P'(x,L_\eta)$. Then $Q(\eta)\downarrow$ by axiom (1), so either $Q(\eta)$ holds (in which case we are done proving set induction), or $\neg Q(\eta)$ holds. In the latter case by axiom (6) we have $\alpha\downarrow$, $\neg Q(\alpha)$, and $\forall \beta < \alpha \; Q(\beta)$, where $\alpha := A_{\neg Q}$. Now take $x\in L_\alpha$. By axiom (5) there exists $\beta < \alpha$ such that $x \subseteq L_\beta$. But then since $Q(\beta)$ holds, we have $P'(y,L_\eta)$ for all $y\in x$, so by hypothesis $P'(x,L_\eta)$ holds. Since $x$ was arbitrary this implies $Q(\alpha)$, a contradiction.

To show $\Delta_0$-separation, let $P(y) \in \Delta_0(\ST)$ and let $x\in L_\eta$. By axiom (5) there exist $\alpha < \eta$, $Q(y,\zz)\in \Delta_0(\ST)$, $M\in\N$, and $\zz\in L_\alpha^M$ such that $x = \{y\in L_\alpha : \eval_\iota(Q,(y,\zz),\alpha) = \top\}$. Then $\{y\in x : P(y)\} = \{y\in L_\alpha : (\eval_\iota(Q,(y,\zz),\alpha) = \top) \wedge P(y)\} \in L_\eta$ by axioms (5) and (8).
\end{proof}

\begin{lemma}[$\AA_\GDST$]
\label{lemmaETS}
If $\eta$ is a limit ordinal greater than $\omega$, then to determine whether or not $\eta$ is admissible, it suffices to check the following alternative to $\Delta_0$-collection:
\begin{itemize}
\item \text{($\Delta_0$-collection version 2)} Given any $\Delta_0$ formula $Q(x,\beta)$ of $\ST$ and any ordinal $\alpha < \eta$, if for every $x\in L_\alpha$ there exists an ordinal $\beta < \eta$ such that $Q'(x,\beta,L_\eta)$ holds, then there exists an ordinal $\gamma < \eta$ such that for every $x\in L_\alpha$ there exists $\beta < \gamma$ such that $Q'(x,\beta,L_\eta)$ holds.
\end{itemize}
\end{lemma}
\begin{proof}
Suppose $P(x,y)$ is a $\Delta_0$ formula of $\ST$, $z\in L_\eta$, and for every $x\in z$ there exists $y$ such that $P(x,y)$ holds. Since $z\in L_\eta$, there exists $\alpha < \eta$ such that $z \subseteq L_\alpha$. Let $Q(x,\beta) = (x\notin z) \vee (\exists y \in L_{\beta} \; P(x,y))$. Then since $\eta$ is a limit ordinal, for all $x\in L_\alpha$, there exists $\beta < \eta$ such that $Q(x,\beta)$ holds. By hypothesis, there exists $\gamma < \eta$ such that for all $x \in L_\alpha$, there exists $\beta < \gamma$ such that $Q(x,\beta)$ holds. The conclusion follows from letting $w = L_\gamma$.
\end{proof}

\section{A canonical interpretation of $\GDST$?}
\label{sectioncanonical}

Since the interpretation $I_\eta$ of $\GDST$ depends on a choice of ordinal $\eta$, it is natural to ask whether there is any specific ordinal with respect to which it is most natural to interpret $\GDST$. Such a choice should be made in accordance with the principle that every mathematical object is describable, but now by ``describable'' we mean ``is the image of a term in the language $\GDST$ under the canonical interpretation''. To this end we make the following definition, which we state in greater generality than immediately necessary for use in the sequel.

\begin{definition}
For each ordinal $\eta$ and for each set $z$, let
\[
f_z(\eta) = \min\{\alpha : \text{there do not exist an ordinal term $A(\ww)$ and $\ww \in z^* := \bigcup_{M = 0}^\infty z^M$ such that $\alpha = j_\eta(A[\ww])$}\}.
\]
We say that $\eta$ is \emph{$z$-descriptionalist} if $f_z(\eta) = \eta$. The following is an immediate consequence of this definition:
\end{definition}

\begin{observation}
\label{observationexistsA}
Fix a $z$-descriptionalist $\theta$. For all $\alpha < \theta$ there exists an ordinal term $A(\ww)$ and $\ww\in z^*$ such that $\alpha = j_{\theta}(A[\ww])$.
\end{observation}

\begin{corollary}
\label{corollaryexistsX}
Fix a $z$-descriptionalist $\theta$. For all $\alpha \leq \theta$ and $x\in L_\alpha$ there exist a set term $X(\ww)$ and $\ww\in z^*$ such that $x = j'_{\theta}(X[\ww])$.
\end{corollary}
\begin{proof}
By induction on $\alpha$. The case where $\alpha$ is zero or a limit ordinal is clear. Suppose the statement is true for $\alpha < \theta$, and we will show it is true for $\alpha + 1$. If $x\in L_{\alpha + 1}$, then there exist $x_1,\ldots,x_M \in L_\alpha$ and $P \in \FF_\GDST(M+1,1)$ defined using no instances of constructor (6) such that $x = j'_\theta(\{y\in L_\alpha : P[x_1,\ldots,x_M,y,\alpha]\})$. By the induction hypothesis there exist terms $X_1,\ldots,X_M$ such that $x_m = j'_{\theta}(X_m[\ww_m])$ for $m = 1,\ldots,M$. By Observation \ref{observationexistsA} there exists a term $A$ such that $\alpha = j_{\theta}(A[\ww_0])$. Thus, $x = j'_{\theta}(\{y\in L_A : P(X_1[\ww_1],\ldots,X_M[\ww_M],y,A[\ww_0])\})$.
\end{proof}

Letting $\alpha = \theta$ gives:

\begin{corollary}
\label{corollaryexistsX2}
Fix a $z$-descriptionalist $\theta$. For all $x\in L_{\theta}$ there exist a set term $X$ and $\ww\in z^*$ such that $x = j'_{\theta}(X[\ww])$.
\end{corollary}

Letting $z = \emptyset$, Corollary \ref{corollaryexistsX} can be interpreted as meaning that if $\theta$ is $\emptyset$-descriptionalist, then $L_{\theta}$ is a model of $\GDST$ every element of which is describable, since $X$ is a description of $x = j_{\theta}'(X)$.

\begin{proposition}
\label{propositionfz}
$f_z$ is weakly increasing, $f_z(\eta) \leq \eta$ (and in particular $f_z(\eta)$ exists) for all $\eta$, and if $z$ is countable then $f_z(\eta) $ is countable for all $\eta$.
\end{proposition}
\begin{proof}
The fact that $f_z$ is weakly increasing follows directly from Observation \ref{observationmonotone}. Since $j_\eta(A[\ww]) < \eta$ or $j_\eta(A[\ww]) = U$ for all $A,\ww$, we get $f_z(\eta) \leq \eta$ for all $\eta$. Finally, $(A,\ww)\mapsto j_\eta(A[\ww])$ is a surjection from the countable set $\bigcup_{M = 0}^\infty \TT_\GDST(M,0)\times z^M$ to $f_z(\eta)$, implying that $f_z(\eta)$ is countable.
\end{proof}

\begin{corollary}
\label{corollarythetaz}
If $\{\eta : f_z(\eta) = \eta\}$ is a set (rather than a proper class), then it has a maximal element, which we denote by $\theta_z$. In particular, this holds if $z$ is countable and there exists an uncountable ordinal.
%Fix a countable set $z$. If there exists an uncountable ordinal, then there exists a maximal ordinal $\theta_z$ such that $f_z(\theta_z) = \theta_z$. Moreover, there exists an integer $K = K_z \geq 1$ such that for all $\eta\geq \theta_z$ we have $f_z^K(\eta) = \theta_z$.
\end{corollary}
\begin{proof}
Indeed, let $\theta_z = \sup\{\eta : f_z(\eta) = \eta\}$. Then by Proposition \ref{propositionfz}
\[
\theta_z = \sup_{\substack{\eta < \theta_z \\ f_z(\eta) = \eta}} \eta = \sup_{\substack{\eta < \theta_z \\ f_z(\eta) = \eta}} f_z(\eta) \leq f_z(\theta_z) \leq \theta_z.
\]
If $z$ is countable and there exists an uncountable ordinal, then $\{\eta : f_z(\eta) = \eta\} \subseteq \omega_1$, where $\omega_1$ is the first uncountable ordinal.
%Let $\omega_1$ be the first uncountable ordinal. Then $f_z(\eta) < \omega_1$ for all $\eta$. Since $f_z$ is monotone, it follows that $f_z^K(\eta) \leq f_z^{K - 1}(\omega_1)$ for all $\eta$ and $K\geq 1$. On the other hand, since $f_z(\eta) \leq \eta$ for all $\eta$, the sequence $(f_z^K(\omega_1))_K$ is a weakly decreasing sequence of ordinals and therefore eventually constant, say $\theta_z := f_z^K(\omega_1) = f_z^{K - 1}(\omega_1)$. Then $\theta_z = f_z^K(\theta_z) \leq f_z^K(\eta) \leq f_z^{K - 1}(\omega_1) = \theta_z$ for all $\eta \geq \theta_z$.
\end{proof}

%\begin{conjecture}
%$\KP + \omega_1$ (or a more powerful system) implies $K_\smallemptyset = 1$.
%\end{conjecture}

\begin{remark}
\label{remarktheta0}
If the hypothesis of Corollary \ref{corollarythetaz} holds for $z = \emptyset$, then $I_{\theta_\smallemptyset}$ is a natural choice for the ``canonical interpretation'' of $\GDST$. By contrast, if the hypothesis fails, then Observation \ref{observationmonotone} shows that there is a limiting interpretation $I = \lim_{\eta\to\infty} I_\eta$ which is descriptionalist in the sense of Corollary \ref{corollaryexistsX2}.

The philosophical content of the dichotomy observed in Corollary \ref{corollarythetaz} requires some explanation. We are using $\KP$ as our metatheory, but in some ways this is a bad choice because $\KP$ does not have a canonical model, and so the truth values of statements like the hypothesis of Corollary \ref{corollarythetaz} depend on the choice of model. (``The entire universe of sets'' is not a good answer to ``what is the canonical model'' from a gradualist point of view, since the universe is always being expanded.)

Thus, $L_{\theta_\smallemptyset}$ is only a ``canonical model'' of $\GDST$ in the sense that the universe of sets is a ``canonical model'' of $\KP$, which is to say, not quite. The following remarks may help illustrate this point:
\begin{itemize}
\item $\theta_\smallemptyset$ is defined as being the largest $\emptyset$-descriptionalist ordinal. The concept of an ordinal being $\emptyset$-descriptionalist is independent of the choice of model as long as we restrict to models of the form $L_\alpha$, but the concept of a ``largest'' such ordinal is not, as it requires quantifying over all ordinals greater than $\theta_\smallemptyset$. The class of ordinals greater than $\theta_\smallemptyset$ is a proper class and is therefore always growing, meaning that quantification over it is not necessarily well-defined.
\item If we assume that there exists an uncountable ordinal, then $\theta_\smallemptyset$ can alternately be defined as the largest $\emptyset$-descriptionalist ordinal less than $\omega_1$, where $\omega_1$ is the smallest uncountable ordinal. This gets rid of the quantification over the proper class of ordinals greater than $\theta_\smallemptyset$, but it introduces a new improper quantification, since an ordinal is defined to be ``uncountable'' if there is no bijection between it an $\N$, requiring us to quantify over the class of bijections, which is again a proper class. Put more concretely, the issue is that an ordinal $\alpha$ may be uncountable with respect to the model $L_\beta$ of $\KP$ for some $\beta > \alpha$, but then countable again with respect to some larger model $L_\gamma$ with $\gamma > \beta$.
\end{itemize}
Nonetheless, we can think of the map
\begin{equation}
\label{canonicaldef}
g:V\mapsto \begin{cases}
L_{\theta_\smallemptyset} & \text{ if $\{\eta : f_\smallemptyset(\eta) = \eta\}$ is a set}\\
L & \text{ otherwise}
\end{cases}
\end{equation}
as a canonical map that sends models of $\KP$ to models of $\GDST$.
\end{remark}

\begin{proposition}
The map \eqref{canonicaldef} is idempotent i.e. satisfies $g^2 = g$.
\end{proposition}
\begin{proof}
If $\{\eta : f_\smallemptyset(\eta) = \eta\}$ is a proper class with respect to $V$, then it is also a proper class with respect to $L$. Conversely, suppose that $\{\eta : f_\smallemptyset(\eta) = \eta\}$ is a set and thus $\theta_\smallemptyset$ exists. For all $\alpha < \theta_\smallemptyset$ we have $\theta_\smallemptyset \neq j_{\theta_\smallemptyset + 1}(\min\{\beta > [\alpha] : F_\smallemptyset(\beta) = \beta\})$, where $F_\smallemptyset(\beta)$ is a term representing $f_\smallemptyset(\beta)$. So there exists $\beta \in (\alpha,\theta_\smallemptyset)$ such that $f_\smallemptyset(\beta) = \beta$. Since $\alpha$ was arbitrary, this shows that $\{\eta : f_\smallemptyset(\eta) = \eta\}$ is a proper class with respect to $L_{\theta_\smallemptyset}$.
\end{proof}

We may now describe how $\GDST$ would answer the two questions posed in the introduction: for any $\eta < \theta_\smallemptyset$ all elements of $L_\eta$ can be thought of as existing simultaneously, but the limit $L_{\theta_\smallemptyset}$ is never reached. Similarly, for $\eta < \theta_\smallemptyset$ and any $\Delta_0$ formula $P(L)$ of $\ST$ the formula $P(L_\eta)$ has a definite truth value (since we can write $\eta = j_{\theta_\smallemptyset}(A)$ for some term $A$ of $\GDST$, and then $\iota_{\theta_\smallemptyset}(P(A)) \in \{\top,\bot\}$).

An alternative the idea that \eqref{canonicaldef} is the ``canonical model'' of $\GDST$ is that for every subset $w$ of the natural numbers, the largest $\{w\}$-descriptionalist ordinal $\theta_{\{w\}}$ is a natural model of $\GDST$ that can be thought of as encoding the information contained in $w$. This turns out to be equivalent to the idea that all countable ordinals are natural models of $\GDST$:

\begin{proposition}
\label{propositionalphathetaw}
For every countable ordinal $\alpha$ there exists $w\subseteq \N$ such that $\alpha < \theta_{\{w\}}$.
\end{proposition}
\begin{proof}
We proceed by induction on $\alpha$.\footnote{This is the one place in the paper where we use full induction rather than $\Delta_0$ induction.} Let $\alpha = \{\beta_n : n\in\N\}$ be a countable ordinal, and for each $n\in\N$ let $w_n \subseteq \N$ be such that $\beta_n < \theta_{\{w_n\}}$. Then by the definition of $\theta_{\{w_n\}}$, there exists a term $A_n \in \TT_\GDST(1,0)$ such that $\beta_n = j_{\theta_{\{w_n\}}}(A_n[w_n])$. Let $w\subseteq\N$ be such that $\iota_n^{-1}(w) = w_n$ and $\Phi(\iota_n'^{-1}(w)) = A_n$, where $\iota_n,\iota_n' : \N \to \N$ are uniformly definable injective maps with disjoint images and $\Phi:\P(\N) \to \TT_\GDST(1,0)$ is a definable surjective map. Finally, let
\[
A(w) = \min\{\alpha : \text{ there does not exist $n$ such that $\alpha = \overline\eval_j(\iota_n'^{-1}(w),\Phi(\iota_n^{-1}(w)))$}\}.
\]
Then $j_{\alpha + 1}(A[w]) = \alpha$, while $j_{\alpha + 1}(A_n[w_n]) = \beta_n$ for all $n$. It follows that $f_{\{w\}}(\alpha + 1) = \alpha + 1$, and thus $\alpha < \theta_{\{w\}}$.
\end{proof}

\section{Reflecting ordinals}
\label{sectionreflecting}

\begin{definition}
An ordinal $\theta$ is \emph{$\tau$-reflecting} for $\tau > \theta$ if there exist a set $z$ and such that $\theta = f_z(\tau)$. Equivalently, $\theta$ is $\tau$-reflecting if $f_\theta(\tau) = \theta$, or equivalently there is no term $A(\ww)$ and $\ww\in \theta^*$ such that $\theta = j_{\tau}(A[\ww])$.

Finally, an ordinal $\theta$ which is $\tau$-reflecting for some $\tau > \theta$ will be called \emph{reflecting}. We note that $\theta$ is reflecting if and only if it is $(\theta + 1)$-reflecting.
\end{definition}

We note that if $\theta_z$ exists, then it is reflecting. Thus, if there exists an uncountable ordinal then a countable reflecting ordinal exists.

Being reflecting is a large ordinal property. It implies being $\Pi_n$-reflecting for every $n\in\N$:

\begin{definition}
\label{definitionST}
We recall that an ordinal $\theta$ is said to be $\Pi_n$-reflecting if for every $\Pi_n$ sentence $P(\ww)$ of $\ST$ and for all $w_1,\ldots,w_M\in L_\theta$, if $L_\theta$ satisfies $P[\ww]$ then so does $L_\alpha$ for some $\alpha < \theta$ such that $w_1,\ldots,w_M \in L_\alpha$. Note that by Remark \ref{remarkQB}, this is equivalent to saying that if $P'[\ww,L_\theta]$ holds, then so does $P'[\ww,L_\alpha]$ for some $\alpha < \theta$ such that $w_1,\ldots,w_M \in L_\alpha$.
%Here we recall that $\ST$ is the standard language of set theory (i.e. the one-sorted first-order predicate logic with binary predicates $\in$ and $=$). Note that neither $\GDST$ nor $\ST$ contains the other, but for any sentence $P(\ww)$ in $\ST$, there is a formula $P'(\ww,L)$ of $\Delta_0(\ST) \subseteq \GDST$ such that for any $\alpha$, $L_\alpha$ satisfies $P[\ww]$ if and only if $\iota_{\alpha + 1}(P'[\ww,L_\alpha]) = \top$.% Moreover, if $P$ is a $\Delta_0$ formula of $\ST$ then we have $P\in \GDST$ as well.
\end{definition}

\begin{proposition}
\label{propositionreflection}
Every reflecting ordinal is $\Pi_n$-reflecting for every $n\in\N$.
\end{proposition}
The converse is also true but the proof is significantly more involved, and since we do not use it in the sequel we omit it.
\begin{proof}
Let $P$ be a formula of $\ST$, and let $P'(L)$ be as in Remark \ref{remarkQB}. Now suppose $\theta$ is reflecting.   If $\theta$ is not $\Pi_n$-reflecting for some $n\in\N$, then for some sentence $P$ in the language of set theory we have $\iota_{\theta + 1}(P'[L_\theta]) = \top$ while $\iota_{\alpha + 1}(P'[L_\alpha]) = \bot$ for all $\alpha < \theta$. Thus $\theta = j_{\theta + 1}(\min\{\eta : \eval_\iota(P'(L_\eta),\eta) = \top\})$, contradicting that $\theta$ is reflecting.
\end{proof}

It is well-known that being $\Pi_n$-reflecting for all $n\in\N$ implies being recursively Mahlo, which implies being admissible (definitions given in Definitions \ref{definitionrecursivelymahlo} below and \ref{definitionKP} above, respectively). For convenience we provide the proofs of these implications below.

\begin{proposition}[Either $\KP$ or $\AA_\GDST$]
\label{propositionadmissible}
Every reflecting ordinal is admissible.
\end{proposition}
In fact, an ordinal is admissible if and only if it is $\Pi_2$-reflecting, see \cite{RichterAczel}, so this proposition follows from \ref{propositionreflection}. For convenience we provide an independent proof.
\begin{proof}
Let $\theta$ be reflecting. First suppose $\theta = \alpha + 1$ for some $\alpha$. Then $\theta = j_{\theta + 1}([\alpha] + 1)$, contradicting that $\theta$ is reflecting.

To show that $\theta$ satisfies $\Delta_0$ collection we apply Lemma \ref{lemmaETS}. Suppose that $\alpha < \theta$ and that $Q(x,\beta)$ is a $\Delta_0$ formula of $\ST$ such that for all $x \in L_\alpha$, there exists $\beta < \theta$ such that $\iota_\theta(Q[x,\beta]) = \top$. Let
\begin{align*}
D(x) &:= \min\{\delta : \exists \beta < \delta \; \eval_\iota(Q[x,\beta],\delta) = \top\}\\
B(x) &:= \min\{\beta : \eval_\iota(Q[x,\beta],D(x)) = \top\}\\
\gamma &:= j_{\theta + 1}(\min\{\gamma : \forall x \in L_{[\alpha]} \;  B(x) < \gamma\}).
\end{align*}
Fix $x \in L_\alpha$. By hypothesis we have $\iota_\theta(Q[x,\beta]) = \top$ for some $\beta < \theta$, which implies $\iota_\delta(Q[x,\beta]) = \top$ for all sufficiently large $\delta < \theta$ by Observation \ref{observationcontinuous}, in particular for some $\delta \in (\beta,\theta)$. It follows that $\iota_\theta(\eval_\iota(P,[(x,\beta)],[\delta]) = \top) = \top$ and thus $\delta := j_\theta(D[x]) \neq U$. By definition, there exists $\beta < \delta$ such that $\iota_\theta(\eval_\iota(P,[(x,\beta)],[\delta]) = \top) = \top$, and thus $\beta := j_\theta(B[x]) \neq U$. By the minimality of $\beta$ we have $\beta < \delta < \theta$.

To summarize, for all $x \in L_\alpha$ we have $j_\theta(B[x]) < \theta$. It follows that $\gamma \leq \theta$, but since $\theta$ is reflecting, we have $\gamma \neq \theta$. So $\gamma < \theta$.
\end{proof}

\begin{corollary}
For all $z$, $\theta_z$ (if it exists) is an admissible ordinal. In particular, $\theta_\smallemptyset$ (if it exists) is an admissible ordinal.
\end{corollary}

We can strengthen this result by recalling the definition of a \emph{recursively Mahlo ordinal}:

\begin{definition}
\label{definitionrecursivelymahlo}
An admissible ordinal $\eta$ is \emph{recursively Mahlo} if for every function $f:\eta\to\eta$ definable in $L_\eta$ by a $\Delta_1$ formula, there is an admissible $\alpha < \eta$ closed under $f$.
\end{definition}

\begin{proposition}
Every reflecting ordinal is recursively Mahlo.
\end{proposition}
\begin{proof}
Let $\theta$ be reflecting, and let $f:\theta\to\theta$ be defined by the $\Delta_1$ formula $P(\alpha,\beta)$. We can assume that $P(
\alpha,\beta)$ is a $\Sigma_1$ formula, while letting $\w P(\alpha,\beta)$ be a $\Sigma_1$ formula equivalent to $\neg P(\alpha,\beta)$. Letting $P'$ and $\w P'$ be as in Remark \ref{remarkQB}, we get that $P(\alpha,\beta)$ holds in $L_\theta$ if and only if $P'(\alpha,\beta,L_{C(\alpha,\beta)})$ holds, where $C(\alpha,\beta) = \min\{\gamma : P'(\alpha,\beta,L_\gamma) \vee \w P'(\alpha,\beta,L_\gamma)\}$. Since $\theta$ is reflecting we have $j_{\theta + 1}(C[\alpha,\beta]) < \theta$ for all $\alpha,\beta < \theta$.

Next, let $B(\alpha) = \min\{\beta : P'(\alpha,\beta,L_{C(\alpha,\beta)})\}$. Since $\theta$ is reflecting we have $j_{\theta + 1}(B[\alpha])) < \theta$ for all $\alpha < \theta$.

Finally, let $D = \min\{\delta : \text{$\delta$ is admissible and for all $\alpha < \delta$, $B(\alpha) < \delta$}\}$. Since $\theta$ is reflecting we have $\eta := j_{\theta + 1}(D) < \theta$. But then $\eta$ is an admissible ordinal closed under $f$.
\end{proof}

\ignore{

\begin{corollary}[Meant to replace the subsequent proposition]
Let $P(\alpha)$ be a ``good'' formula. Then the formula
\[
P'(\alpha) \equiv \text{ $\alpha$ is a limit of $\beta$ such that $\iota_\alpha(P[\beta]) = \top$ }
\]
is good.
\end{corollary}

\begin{proposition}
\label{propositionsequence}
Let $\tau$ be an ordinal such that $f_z(\tau) < \tau$, and let $P(\ww,\alpha)$ be a formula such that
\begin{itemize}
\item $\iota_\tau(P[\ww,\alpha]) \neq U$ for all $\alpha \leq f_z(\tau)$ and $\ww\in z^*$ and
\item $\iota_\tau(P[\ww,f_z(\alpha)]) = \top$ for all $\alpha \leq \tau$ such that $f_z(\alpha) < \alpha$.
\end{itemize}
Let $g_P(\alpha)$ denote the $\alpha$th ordinal such that $\iota_\tau(P[\alpha]) = \top$ (noting that $g_P(\alpha)$ may be undefined). Then $g_P(f_z(\tau)) = f_z(\tau)$. In particular $g_P(\theta_z) = \theta_z$, and $f_z(g_P(\theta_z + 1)) = \theta_z$ if $g_P(\theta_z + 1)$ is defined.
\end{proposition}
\begin{proof}
%First note that by letting $T(w) = \min\{\alpha : P(\alpha) \wedge \forall (\beta,\gamma)\in w \; \gamma < \alpha\}$ and applying Lemma \ref{lemmarecursion}, there exists a formula $G_P(\alpha) = T^\alpha$ such that $\iota_\tau(G_P(\alpha)) = g_P(\alpha)$ for all $\alpha$ such that $g_P(\alpha) < \tau$.
Write $f_z(\tau) = g_P(\alpha)$ for some $\alpha \leq f_z(\tau)$. If $\alpha = f_z(\tau)$ we are done showing $g_P(f_z(\tau)) = f_z(\tau)$, so suppose $\alpha < f_z(\tau)$ and thus $\alpha = j_\tau(A[\ww])$ for some $A(\ww)$ and $\ww\in z^*$. Since $g_P(\alpha) = f_z(\tau) < \tau$, we have $f_z(\tau) = j_\tau(G_P(A[\ww]))$, where $G_P(\alpha)$ is a term representing the function $g_P$ (cf. Corollary \ref{corollaryGP}). This contradicts the definition of $f_z$. So $g_P(f_z(\tau)) = f_z(\tau)$.

Letting $\tau = \theta_z + 1$, we get $g_P(\theta_z) = \theta_z$. Finally, suppose $\tau := g_P(\theta_z + 1)$ exists, and observe that $f_z(\tau) = g_P(\alpha)$ for some $\alpha$. Since $f_z(\tau) < \tau = g_P(\theta_z + 1)$ we get $\alpha < \theta_z + 1$ i.e. $\alpha \leq \theta_z$ and thus $f_z(\tau) = g_P(\alpha) \leq g_P(\theta_z) = \theta_z$.
\end{proof}

\begin{example}
\label{examplesequence}
The hypotheses of Proposition \ref{propositionsequence} are satisfied when
\begin{itemize}
\item $P_0(\beta) \equiv \text{($\beta$ is admissible)}$ by Proposition \ref{propositionadmissible}
\item $P_{\alpha + 1}(\beta) \equiv (g_{P_\alpha}(\beta) = \beta)$ where $P_\alpha$ is as in Proposition \ref{propositionsequence}.
\end{itemize}
\end{example}

In what follows we use the shorthand $g_\alpha = g_{P_\alpha}$.

\begin{corollary}
The Church--Kleene ordinal $\omega_{CK}$ satisfies $\omega_{CK} < \theta_\smallemptyset$.
\end{corollary}
\begin{proof}
$\omega_{CK} = g_0(1)$, so since $g_0(\theta_\smallemptyset) = \theta_\smallemptyset$ we have $\omega_{CK} < \theta_\smallemptyset$.
\end{proof}

%Note: The reason we can't use $\Sigma_1$ instead of $\GDST$ is that $\alpha =_\eta \min\{\beta : P_\eta(\beta)\}$ requires $P_\eta(\beta)$ to be false for $\beta < \alpha$, which is not a $\Sigma_1$ statement. In the language $\GDST$, both the truth and falsity (but not the undefinedness) of $P_\eta$ are $\Sigma_1$ statements.

}% end ignore

\ignore{

\section{Lower bounds for $\theta_z$}

%So far we have not given an axiomatization of $\GDST$. If $A$ is a term such that $\alpha = j_{\theta_z}(A)$ is an admissible ordinal, then we can embed $\ST$ into $\GDST$ by replacing unbounded quantifiers $\forall x$ and $\exists x$ with quantifiers $\forall x\in L_A$ and $\exists x\in L_A$ restricted to $L_\alpha$. Denoting this embedding by $\ST_A$, we can use $\KP$ as an axiomatization of $\ST_A$. In this section, we explore the possibility of using stronger axiomatizations of $\ST_A$ for certain terms $A$, justified by reference to our interpretation of $\GDST$ within $\KP + \omega_1$.

\begin{lemma}
There exists an ordinal term $F = F_z(\alpha)$ with ordinal input $\alpha$ and set input $z$ such that for all $\eta < \tau$ and $z$, we have $j_\tau(F_{[z]}[\eta]) = f_z(\eta)$. Similarly, there exists an ordinal term $\overline F = \overline F_z$ such that for all $\tau$ and $z$, we have $j_\tau(\overline F_{[z]}) = f_z(\tau)$ if $f_z(\tau) < \tau$, and $j_\tau(\overline F_{[z]}) = U$ otherwise.
\end{lemma}
\begin{proof}
Follows directly from Proposition \ref{propositionselfmeta} and Corollary \ref{corollaryselfmeta} respectively.
\end{proof}

\begin{proposition}
Let $P(\theta)$ be a formula such that $\iota_\tau(P[\theta]) = \top$ for some $\theta = f_z(\tau) < \tau$. Then there exists $\alpha < \theta$ such that $\iota_\tau(P[\alpha]) = \top$.
\end{proposition}
\begin{proof}
Otherwise, $\theta = j_\tau(\min\{\beta : \eval_\iota(P[\beta]) = \top\})$, contradicting the definition of $f_z$.
\end{proof}

\begin{proposition}
For all $\alpha < \theta = f_z(\tau) < \tau$, there exists $\alpha \leq \eta < \theta$ such that $f_z(\eta) = \eta$.
\end{proposition}
\begin{proof}
If not, then
\[
\theta = j_\tau(\min\{\beta : F_{[z]}(\beta) = \beta > A\}),
\]
where $A$ is an ordinal term such that $j_\tau(A) = \alpha$ (such a term exists because $f_z(\tau) = \theta > \alpha$). This contradicts the definition of $f_z$.
\end{proof}

\begin{proposition}
For all $\alpha_0 < \theta = f_z(\tau) < \tau$ there exists $\alpha_0 \leq \alpha < \theta$ such that $g_\alpha(\alpha) = \alpha$.
\end{proposition}
\begin{proof}
Otherwise $\theta = j_\tau(\min\{\alpha : g_\alpha(\alpha) = \alpha\})$, contradicting the definition of $f_z$.
\end{proof}

\begin{proposition}
\label{propositionfgfixedpoint}
Let $G(\alpha)$ be an ordinal term depending on an ordinal input $\alpha$. Let $z$ be a finite set and let $\theta = f_z(\theta) = f_z(\tau) < \tau$ and suppose that $g(\theta) := j_\tau(G[\theta]) \neq U$.  Suppose also that for all $\alpha < \theta$, we have $g(\alpha) := j_{\theta}(G[\alpha]) \neq U$. Then for all $\alpha_0 < \theta$ there exists an ordinal $\alpha_0 \leq \alpha < \theta$ such that $\alpha = f_z(g(\alpha))$.
\end{proposition}
\begin{proof}
Suppose the conclusion fails, and let $\alpha_0$ be as in the conclusion. By Observation \ref{observationexistsA} we have $\alpha_0 = j_{\theta}(A[\ww])$ for some $A(\ww)$ and $\ww\in z^*$. Since $\theta \leq g(\theta) < \tau$ we have $\theta = f_z(\theta) \leq f_z(g(\theta)) \leq f_z(\tau) = \theta$ and thus $f_z(g(\theta)) = \theta$, so
\[
\theta = j_\tau(\min\{ \alpha \geq A : \alpha = F_{[z]}(G(\alpha)) \})
\]
which contradicts $f_z(\tau) = \theta$ by the definition of $f_z$.
\end{proof}

\begin{corollary}
\label{corollaryadmissiblesequence}
Let $z$ be a finite set, and let $\theta_z$ be as in Corollary \ref{corollarythetaz}. For all $\alpha_0 < \theta_z$, there exists $\alpha_0 \leq \alpha < \theta_z$ such that $f_z(\beta) = \alpha$ for an admissible $\beta > \alpha$, and $g_0(\beta) = \beta$.
\end{corollary}
\begin{proof}
Let the functions $g_0,g_1,g_2$ be as in Example \ref{examplesequence}, and note that $g_2(\theta_{z\cup\{\theta_z\}}) = \theta_{z\cup\{\theta_z\}}$ by Proposition \ref{propositionsequence}. Since $\theta_{z\cup\{\theta_z\}} > \theta_z$, this implies that $\tau := g_2(\theta_z + 1)$ exists. Let $\tau = \theta_z + 1$, and observe that $f_z(\tau) = \theta_z$. Let $G(\alpha) = G_1(\alpha + 1)$, where $G_1(\alpha)$ is a term representing $g_1(\alpha)$. By Proposition \ref{propositionfgfixedpoint} we have $\alpha = f_z(\beta)$, where $\beta := g_1(\alpha + 1)$ for some $\alpha < \theta_z$. Now, $\beta > \alpha$ and $g_0(\beta) = \beta$ and $f_z(\beta) = \alpha$, which completes the proof.
\end{proof}

This suggests an axiomatization of $\GDST$ as follows: letting $B(\alpha_0) = \min\{ \beta > \alpha_0 : (g(\beta) = \beta) \wedge (f_z(\beta) < \beta)\}$ and $A(\alpha_0) = F(B(\alpha_0))$, for each term $A$ such that $A\downarrow$ axiomatize $\ST_{B(A)}$ by $\KP$ together with the axioms $\forall \alpha < B(A) \; f_z(\alpha) \leq A$ and $\forall \alpha < B(A) \; G(\alpha)\downarrow$.

%Note: If $S \subseteq \N$ then we can define $\theta_z(S)$ analogously; it is possible that $\theta_z(S) > \theta_z$. [Philosophical implications\internal]

%\begin{proposition}
%Let $t(\alpha)$ be an ordinal term depending on an ordinal variable $\alpha$, and suppose that
%\begin{itemize}
%\item[(a)] there exists $\tau > \theta_z$ such that the sequence $\tau_0 = \tau$, $\tau_{k+1} = j_{\tau_k}(t_k[\theta_z])$ is well-defined for $k=0,\ldots,K-1$
%\item[(b)] for all $\alpha < \theta_z$, the sequence $\alpha_0 = \alpha$, $\alpha_{k+1} := j_{\theta_z}(t_k,\alpha_k)$ is well-defined for $k=0,\ldots,K-1$
%\end{itemize}
%Then there exists $\alpha < \theta_z$ such that $\eta \neq j_{\alpha_K}(s,\emptyset)$ for all ordinal terms $s$.
%\end{proposition}

}% end ignore

\section{The language of not necessarily monotonic inductive definitions}
\label{sectionNMID}

%\begin{proposition}
%If $\eta$ is the running time of a ID based on a set $z$, then $\eta < \theta_z$.
%\end{proposition}
%\begin{proof}
%Let $\tau = \eta + 1$. We claim that $f_z(\tau) = \tau$, which implies $\eta < \tau \leq \theta_z$. Indeed, fix $\alpha < \tau$, and we will show that $j_\tau(T) = \alpha$ for some term $T \in \TT_\GDST(0,0)$. First suppose that $\alpha < \eta$, and let $n\in S_{\alpha + 1} \butnot S_\alpha$. Then we can let $T = \min\{\beta : n\in S_{\beta + 1}\}$. Next, suppose $\alpha = \eta$. Then we can let $T = \min\{\beta : S_{\beta + 1} = S_\beta\}$.
%\end{proof}

%Question: Is the function $g(\alpha) =$ $\alpha$th admissible ordinal describable relative to ID?\internal

The remainder of the paper will be devoted to stating and proving an equivalence between a large ordinal axiom defined in terms of the functions $f_z$ (the existence of a ``$k$-fold reflecting ordinal''), and an assumption about the existence of limits of a certain class of not necessarily monotonic inductive definitions.

\ignore{

\begin{definition}
\label{definitionkfold}
An ordinal $\eta$ is \emph{$\alpha$-fold reflecting} if there exist $\eta = \eta_0 < \ldots < \eta_\alpha$ such that for $\beta < \alpha$, we have $f_{\eta_\beta}(\eta_\alpha) = \eta_\beta$.
\end{definition}

\begin{proposition}
Suppose that an uncountable ordinal exists. Then for all $\alpha < \theta_\smallemptyset$, $L_{\theta_\smallemptyset}$ satisfies ``there exist arbitrarily large $\alpha$-fold reflecting ordinals''.
%Fix an integer $k$. There exist $\eta_0 < \ldots < \eta_k < \theta_\smallemptyset$ such that if $z_i = \{\eta_0,\ldots,\eta_{i - 1}\}$, then $f_{z_i}(\eta_{i + 1}) = \eta_i$ for all $i = 0,\ldots,k - 1$.
\end{proposition}
\begin{proof}
Define inductively $\sigma_\beta = f_{z_\beta}(\omega_1)$ for all $\beta \leq \alpha$, where $z_\beta = \{\sigma_\gamma : \gamma < \beta\}$. Then $\sigma_\beta$ is countable for all countable $\beta$. Finally, let $\tau = \omega_1$. Then $\sigma_0 < \ldots < \sigma_\alpha < \tau$. Let
\begin{equation}
\label{Palphatau}
P(\alpha,\tau) \equiv \exists g\in L_\tau \forall \beta < \alpha \; (f_{g(\beta)}(\tau) = g(\beta)) \wedge (\forall \gamma < \beta < \alpha) (g(\gamma) < g(\beta))
\end{equation}
Note that $P(\alpha,\omega_1) = \top$, so $\eta := \min\{\tau : \exists \alpha < \tau \; P(\alpha,\tau)\} \neq U$. Let $g\in L_\eta$ be as in \eqref{Palphatau}. By the minimality of $\eta$ the order type of $\{\gamma < \eta : f_\gamma(\eta) = \gamma\}$ is $\alpha$, as otherwise we could replace $\eta$ by $g(\alpha + 1) < \eta$ and still satisfy $P(\alpha,\eta)$. For each such $\gamma$, we have $\gamma = j_{\eta + 1}(g(\beta))$, so for all $\beta \leq \eta$ we have $f_\beta(\eta + 1) > \beta$. An induction argument then shows that $f_\smallemptyset(\eta + 1) = \eta + 1$, so $\eta + 1 \leq \theta_\smallemptyset$.
\end{proof}
}% end ignore

\begin{definition}
\label{definitionkfold}
An ordinal $\eta$ is \emph{$k$-fold reflecting} if there exist $\eta = \eta_0 < \ldots < \eta_k$ such that for $i = 0,\ldots,k - 1$, we have $f_{\eta_i}(\eta_{i + 1}) = \eta_i$.
\end{definition}

\begin{proposition}
Suppose that an uncountable ordinal exists. Then for all $k\in\N$, $L_{\theta_\smallemptyset}$ satisfies ``there exist arbitrarily large $k$-fold reflecting ordinals''.
%Fix an integer $k$. There exist $\eta_0 < \ldots < \eta_k < \theta_\smallemptyset$ such that if $z_i = \{\eta_0,\ldots,\eta_{i - 1}\}$, then $f_{z_i}(\eta_{i + 1}) = \eta_i$ for all $i = 0,\ldots,k - 1$.
\end{proposition}
\begin{proof}
Define inductively $z_0 = \emptyset$, $\theta_i = \theta_{z_i}$, and $z_{i + 1} = z_i \cup \{\theta_i\}$, all of which exist by Corollary \ref{corollarythetaz}. Since $\theta_i \in z_{i + 1}$ we have $\theta_i \neq f_{z_{i + 1}}(\theta_i + 1)$, and since $z_i \subseteq z_{i + 1}$ we have $f_{z_{i + 1}}(\theta_i + 1) \geq f_{z_i}(\theta_i) = \theta_i$, so $f_{z_{i + 1}}(\theta_i + 1) = \theta_i + 1$. It follows from the definition of $\theta_{i + 1}$ that $\theta_i + 1 \leq \theta_{i + 1}$, or equivalently $\theta_i < \theta_{i + 1}$.

We claim that $f_{z_i}(\theta_{i + 1}) = \theta_i$ for all $i$. Indeed, $f_{z_i}(\theta_{i + 1}) \geq f_{z_i}(\theta_i) = \theta_i$, so it suffices to show that $j_{\theta_{i + 1}}(A[\theta_0,\ldots,\theta_{i - 1}]) \neq \theta_i$ for all $A\in\TT_\GDST(0,i)$ By contradiction suppose $j_{\theta_{i + 1}}(A[\ww]) = \theta_i$. Since $\theta_{i + 1} = f_{z_{i + 1}}(\theta_{i + 1})$, for all $\beta < \theta_{i + 1}$ there exist $B\in \TT_\GDST(0,i + 1)$ such that $j_{\theta_{i + 1}}(B[\theta_0,\ldots,\theta_i]) = \beta$. It follows that $j_{\theta_{i + 1}}(B([\theta_0],\ldots,[\theta_{i - 1}],A[\theta_0,\ldots,\theta_{i - 1}])) = \beta$, so $f_{z_i}(\theta_{i + 1}) = \theta_{i + 1}$, contradicting the maximality of $\theta_i$.

Now let
\[
P(\beta,\alpha_0,\ldots,\alpha_k) \equiv \bigwedge_{i=0}^{k - 1} \big(F_{\{\alpha_0,\ldots,\alpha_{i - 1}\}}(\alpha_{i + 1}) = \alpha_i < \alpha_{i + 1}\big) \wedge \big(F_{\{\alpha_0,\ldots,\alpha_{k - 1}\}}(\alpha_k + 1) = \alpha_k\big) \wedge (\alpha_0 > \beta)
\]
and reverse-inductively for $i = k,\ldots,0$ let
\[
A_i(\beta) = \min\{\alpha : \exists \alpha_0,\ldots,\alpha_{i - 1} < \alpha \text{ satisfying } P(\beta,\alpha_0,\ldots,\alpha_{i - 1},\alpha,A_{i + 1},\ldots,A_k)\}
\]
Then $\iota_{\theta_k + 2}(P[\beta,\theta_0,\ldots,\theta_k]) = \top$, so $\eta_k := j_{\theta_k + 2}(A_k[\beta]) \neq U$. It is easy to see that this implies $\eta_k = j_{\eta_k + 2}(A_k)$, and by reverse induction we have $\eta_i := j_{\eta_k + 2}(A_i[\beta]) \neq U$ for $i = 0,\ldots,k - 1$; moreover, $\iota_{\eta_k + 2}(P[\beta,\eta_0,\ldots,\eta_k]) = \top$. So $\eta_k + 1 = f_{\{\eta_0,\ldots,\eta_k\}}(\eta_k + 1) = f_\smallemptyset(\eta_k + 1)$, which implies that $\eta_k < \eta_k + 1 \leq \theta_\smallemptyset$. On the other hand, we have $\iota_{\eta_k + 2}(P[\beta,\eta_0,\ldots,\eta_k]) = \top$, so $f_{\{\eta_0,\ldots,\eta_{i - 1}\}}(\eta_{i + 1}) = \eta_i$ for $i = 0,\ldots,k - 1$. Thus, $L_{\theta_\smallemptyset}$ satisfies that $\eta_0$ is a $k$-fold reflecting ordinal greater than $\beta$.
\end{proof}

%\begin{proof}[Alternate proof]
%
%Let
%\begin{align*}
%A_k(\alpha_0,\ldots,\alpha_{k - 1}) &= \min\{ \alpha : F_{\{\alpha_0,\ldots,\alpha_{k - 1}\}}(\alpha + 1) = \alpha\}\\
%A_{i,i}(\alpha_0,\ldots,\alpha_{i - 1}) = A_i(\alpha_0,\ldots,\alpha_{i - 1}) &= \min\{ \alpha : F_{\{\alpha_0,\ldots,\alpha_{i - 1}\}}(A_{i + 1}(\alpha_0,\ldots,\alpha_{i - 1},\alpha)) = \alpha\}\\
%A_{i,j}(\alpha_0,\ldots,\alpha_{i - 1}) &= A_{i + 1,j}(\alpha_0,\ldots,\alpha_{i - 1},A_i(\alpha_0,\ldots,\alpha_{i - 1}))
%\end{align*}
%Then for all $\alpha_0 < \ldots < \alpha_{k - 1} < \theta_0$ we have $j_{\theta_1}(A_k[\alpha_0,\ldots,\alpha_{k - 1}]) \leq \theta_0$, and since $f_\smallemptyset(\theta_1) = \theta_0$ we have $j_{\theta_1}(A_k[\alpha_0,\ldots,\alpha_{k - 1}]) < \theta_0$.
%
%Suppose $i = 0,\ldots,k - 1$ is such that $j_{\theta_1}(A_{i + 1,j}[\alpha_0,\ldots,\alpha_i]) < \theta_0$ for all $j = i + 1,\ldots,k$ and $\alpha_0 < \ldots < \alpha_i < \theta_0$.
%\end{proof}

\begin{definition}
Fix $\AA \subseteq \P(\N)$,\footnote{In the framework of $\KP$, $\P(\N)$ may be a proper class; what we mean is to fix a set $\AA$ such that for all $w\in \AA$ we have $w\subseteq \N$.} and let  $t:\AA \to \AA$ be a map such that $t(w) \supseteq w$ for all $w\in\AA$, and for $w_0\in\AA$ consider the family of sets $(t^\alpha(w_0))_\alpha$ defined recursively as follows:
\begin{itemize}
\item Let $t^0(w_0) = w_0$. 
\item At any given successor ordinal stage $\alpha + 1$, suppose $t^\alpha(w_0)\in\AA$ is defined. Let $t^{\alpha + 1}(w_0) = t(t^\alpha(w_0))$.
\item At any given limit ordinal stage $\alpha$, let $t^\alpha(w_0) = \bigcup_{\beta < \alpha} t^\beta(w_0)$. Note that \emph{a priori} we do not know whether $t^\alpha(w_0)$ is in $\AA$.
\item Finally, let $t^\infty(w_0) = \bigcup_\alpha t^\alpha(w_0)$, assuming this union exists.
\end{itemize}
We say that $t$ \emph{acts $\AA$-regularly on $w_0$} if $t^\alpha(w_0) \in \AA$ is defined for all $\alpha$. If in addition $t^\infty(w_0)\in\AA$ is defined, we say that $t$ \emph{acts infinitely $\AA$-regularly on $w_0$}.
\end{definition}

\begin{observation}
\label{observationinfinitelyregularly}
Assume either that there exists an uncountable ordinal, or that $\Sigma_1$-separation holds. Then every map that acts $\AA$-regularly on $w_0$ also acts infinitely $\AA$-regularly on $w_0$, and moreover for each such map $t$ there exists a countable ordinal $\alpha$ such that $t^\beta(w_0) = t^\infty(w_0)$ for all $\beta\geq \alpha$.
\end{observation}
\begin{proof}
It suffices to show that $t^{\alpha + 1}(w_0) \setminus t^\alpha(w_0) = \emptyset$ for some $\alpha$. Indeed, there is a $\Delta_0$ formula $P(n,\alpha,x)$ such that $n\in x = t^{\alpha + 1}(w_0) \setminus t^\alpha(w_0)$ if and only if $P(n,\alpha,x)$ holds. So if $\Sigma_1$-separation holds, then the set $\{n : \exists \alpha,x \; P(n,\alpha,x)\}$ exists, and thus by $\Delta_0$-collection the set $\{\alpha : \exists n,x \; P(n,\alpha,x)\}$ exists. Since the ordinals are a proper class this implies that there is some $\alpha$ not in this set, which yields the conclusion.

On the other hand, if there exists an uncountable ordinal $\omega_1$ and the conclusion fails, then $\Delta_0$-collection proves the existence of a function $f:\omega_1 \to \N$ such that $\exists x\; P(f(\alpha),\alpha,x)$ holds for all $\alpha$. It follows that $f$ is injective, contradicting the uncountability of $\omega_1$.
\end{proof}

Although Observation \ref{observationinfinitelyregularly} is illustrative, from a gradualist descriptionalist point of view neither of its possible assumptions are valid. Instead, we make the following intuitive argument to establish its conclusion in a way consistent with gradualist descriptionalist principles: Consider the process defining the transfinite sequence $(t^\alpha(w_0))_\alpha$. At any stage $\alpha$ such that $t^{\alpha + 1}(w_0) \setminus t^\alpha(w_0) \neq \emptyset$, we can give $\alpha$ a natural number label chosen from the set $t^{\alpha + 1}(w_0) \setminus t^\alpha(w_0)$. This label is unique to $\alpha$, so the existence of $\alpha$ does not contradict descriptionalism. Since the set $t^\alpha(w_0)$ increases with respect to $\alpha$, eventually the process will ``run out of room'' and be forced to satisfy $t^{\alpha + 1}(w_0) \setminus t^\alpha(w_0) = \emptyset$.

This idea can be made more concrete in the following way: Imagine a world in which there is an infinite array of lightswitches, the $n$th of which is initally set in either the ``off'' or ``on'' position depending on whether $n\in w_0$. The laws of physics of this universe are such that whenever no lightswitch is in the process of switching from ``off'' to ``on'', if $w$ is the set of all lightswitches currently in the ``on'' position then all lightswitches corresponding to elements of $t(w)\setminus w$ begin switching from ``off'' to ``on'', a process which takes $2^{-\min(t(w)\setminus w)}$ seconds. It is easily seen that the total amount of time the process takes before halting is $\leq 1$ second, proving that the process does indeed halt.

\begin{proposition}
\label{propositionTxalpha}
Let $\theta = f_\theta(\tau) < \tau$ with $\tau$ reflecting (so that $\theta$ is 2-fold reflecting), and let $T_\xx(w)$ be a set term of $\GDST$ with a set input $w$ and other inputs $\xx$. Suppose that for all $w\in \AA_\theta := L_\theta \cap \P(\N)$ and $\xx\in L_\theta^M\times \theta^N$, we have $w\subseteq j'_\theta(T_{[\xx]}[w])\in \P(\N)$ (in particular $j_\theta(T_{[\xx]}[w]) \neq U$), and that the same is also true with $\theta$ replaced by $\tau$. Consider the function $t_\xx: \AA_\theta \to \AA_\theta$ defined by the formula
\[
t_\xx(w) = j'_\theta(T_{[\xx]}[w]).
\]
Then for all $\xx\in L_\theta^M \times \theta^N$ and $w_0\in \AA_\theta$, the map $t_\xx$ acts infinitely $\AA_\theta$-regularly on $w_0$. Even stronger, there exist formulas $T_\xx^\alpha(w_0)$ an $T_\xx^\infty(w_0)$ of $\GDST$ such that
\begin{align} \label{Txalpharecursive}
j'_\theta(T_{[\xx]}^{[\alpha]}[w_0]) &= t_\xx^\alpha(w_0) \in \AA_\theta\\ \label{Txinfty}
j'_\theta(T_{[\xx]}^\infty[w_0]) &= t_\xx^\infty(w_0) \in \AA_\theta
\end{align}
for all $\alpha < \theta$, $w_0\in \AA_\theta$, and $\xx\in L_\theta^M \times \theta^N$.
\end{proposition}
\begin{proof}
For an ordinal $\eta < \theta$, we call a map $h:\eta\to \P(\N)$ \emph{$(T_\xx,\theta,w_0)$-compatible} if $h(0) = w_0$, $h(\alpha) = \bigcup_{\beta < \alpha} h(\beta)$ for limit $\alpha < \eta$, and $h(\alpha + 1) = j'_\theta(T_{[\xx]}[h(\alpha)])$ for all $\alpha +  1 < \eta$. Note that there is a formula $P(h) = P_T(h,w_0,\xx)$ such that $h\in L_\theta$ is $(T_\xx,\theta,w_0)$-compatible if and only if $\iota_\theta(P[h,w_0,\xx]) = \top$. Now let
%If $\eta < \theta$, we call a function $h:\eta \to L_\theta$ \emph{$(T,\theta)$-compatible} if
%\[
%h(\alpha) = j'_\theta(T[h\given \alpha,\xx]) \text{ for all } \alpha < \eta.
%\]
%It is not hard to show that there is a formula $P(h,\xx)$ such that $\iota_\theta(P[h,\xx]) = \top$ if and only if $h$ is $(T,\theta)$-compatible. We now define
\begin{align*}
H_{w_0,\xx,\eta} &:= \{(\alpha,y) : \exists h\in L_\eta \text{ such that } P(h,w_0,\xx) \text{ and } h(\alpha) = y\}\\
T_\xx^\alpha(w_0) &:= H_{\xx,\min\{\eta : H_{w_0,\xx,\eta}(\alpha)\downarrow\}}(\alpha).
\end{align*}
(Here $H(\alpha)$ denotes function application, so that $j'_\theta(H[\alpha])$ is the image of $\alpha$ under the function $j'_\theta(H)$, rather than denoting that the $\GDST$ term $H$ has $\alpha$ as a free variable. So $H(\alpha)\downarrow$ means ``the function $H$ is well-defined at $\alpha$''.)

We need to show that \eqref{Txalpharecursive} holds for all $\alpha < \theta$. We proceed by induction; suppose it holds for all $\beta < \alpha$, for some $\alpha < \theta$. It follows from the definition of $T_\xx^\alpha$ that
\[
\iota_\tau(\forall \beta < [\alpha] \;\exists \eta < [\theta] \; H_{[w_0,\xx],\eta}(\beta)\downarrow ) = \top
\]
and by the definition of $f_\theta$ together with the fact that $\theta = f_\theta(\tau) < \tau$, there exists $\eta < \theta$ such that
\[
\iota_\tau(\forall \beta < [\alpha] \;\exists \eta' < [\eta] \; H_{[w_0,\xx],\eta'}(\beta)\downarrow ) = \top
\]
and thus
\[
\iota_\theta(\forall \beta < [\alpha] \; H_{[w_0,\xx,\eta]}(\beta)\downarrow )
= \iota_\tau(\forall \beta < [\alpha] \; H_{[w_0,\xx,\eta]}(\beta)\downarrow ) = \top
\]
from which it follows that $h_{w_0,\xx}\given \alpha = j'_\theta(H_{[w_0,\xx,\eta]}) \cap (\alpha\times L_\theta) \in L_\theta$, where $h_{w_0,\xx}(\beta) := j'_\theta(T_{[\xx]}^{[\beta]}[w_0])$ for all $\beta < \theta$ such that this expression is defined. Now let $y = t_\xx^\alpha(w_0) = j'_\theta( \bigcup_{\beta < [\alpha]} T_{[\xx]}(H_{[w_0,\xx,\eta]}(\beta))) \in \AA_\theta$. The function
\[
h := ( h_{w_0,\xx}\given\alpha ) \cup \{(\alpha,y)\} \in L_\theta
\]
is $(T_\xx,\theta,w_0)$-compatible and thus $h \subseteq j'_\theta(H_{[w_0,\xx,\eta]}) \subseteq h_{w_0,\xx}$ for some $\eta < \theta$. Thus $h_{w_0,\xx}(\alpha) = y$, which completes the proof of \eqref{Txalpharecursive}.

We proceed to prove \eqref{Txinfty}. Note that by our assumptions on $\tau$,
\begin{align*}
j'_\tau\left(T_{[\xx]}^{[\theta]}[w_0]\right) = j'_\tau\left(\bigcup_{\alpha < [\theta]} T_{[\xx]}^\alpha[w_0]\right) &\in L_\tau \cap \P(\N)\\
j'_\tau\left(T_{[\xx]}^{[\theta + 1]}[w_0]\right) = j'_\tau\left(T_{[\xx]}\big[j'_\tau(T_{[\xx]}^{[\theta]}[w_0])\big]\right) &\in L_\tau \cap \P(\N).
\end{align*}
Since by assumption $\theta = f_\theta(\tau) < \tau$, for each $n \in \N$ we have
\begin{align*}
n\notin j'_\tau\left(T_{[\xx]}^{[\theta + 1]}[w_0]\right) \setminus j'_\tau\left(T_{[\xx]}^{[\theta]}[w_0]
\right)
\end{align*}
since otherwise $\theta = j_\tau(\min\{\alpha : [n]\in T_{[\xx]}^{\alpha + 1}[w_0] \setminus T_{[\xx]}^\alpha[w_0]\})$. It follows that
\[
j'_\tau\left(T_{[\xx]}^{[\theta + 1]}[w_0])\right) = j'_\tau\left(T_{[\xx]}^{[\theta]}[w_0]\right)
\]
and so since $\theta = f_\theta(\tau) < \tau$, there exists $\alpha < \theta$ such that
\[
j'_\tau\left(T_{[\xx]}^{[\alpha + 1]}[w_0])\right) = j'_\tau\left(T_{[\xx]}^{[\alpha]}[w_0]\right)
\]
and thus
\[
j'_\theta\left(T_{[\xx]}^{[\alpha + 1]}[w_0])\right) = j'_\theta\left(T_{[\xx]}^{[\alpha]}[w_0]\right).
\]
Letting
\[
T_\xx^\infty(w_0) = T_\xx^{\min\{\alpha : T_\xx^{\alpha + 1}(w_0) = T_\xx^{\alpha}(w_0)\}}(w_0)
\]
finishes the proof.
\end{proof}

%\begin{proposition}
%There exist $\eta_k$ increasing to $\theta = f_\theta(\tau) < \tau < \theta_\smallemptyset$ such that $f_{\eta_k}(\theta) = \eta_k$.
%\end{proposition}
%\begin{proof}
%Let $\theta_k = f_{w_k}(\omega_1)$, where $w_0 = \emptyset$ and $w_{k + 1} = w_k \cup \{\theta_k\}$. Then let
%\[
%P(\eta) \equiv f_\eta(\eta + 1) = \eta \wedge \exists \theta \leq \eta \;\; \forall \alpha_0 < \theta \;\; \exists \alpha \in (\alpha_0,\theta), \; z \;\; f_z(\theta) = \alpha
%\]
%We have $\iota_{\omega_1 + 2}(P[\omega_1]) = \top$, so
%\[
%\eta := j_{\omega_1 + 2}(\min\{\alpha : P(\alpha)\}) \neq U.
%\]
%Let $\theta_k = j_\eta(\min\{\alpha > \theta_{k - 1} : \overline F_\alpha = \alpha\})$, and let $\theta = \sup \theta_k \leq \eta$. Clearly, $f_{\theta_k}(\theta) = \theta_k$ for all $k$. If $\theta < f_\theta(\eta)$, then for some $k$ we have $\theta < f_{\{\theta_0,\ldots,\theta_k\}}(\eta)$, but $\theta_k = j_\eta(\min\{\alpha > \theta_{k - 1} : F_\alpha[\theta] = \alpha, []\})$ [Cont\internal]
%
%\end{proof}

\begin{corollary}
\label{corollaryID}
Consider the sublanguage $\NMID \subseteq \GDST$ generated by production rules \text{(1-5)}, and \text{(6b)}:
\begin{itemize}
\item[(6b)] If $T_\xx(w)\in \TT_\NMID(M + 1,N)$ is a term depending on a set variable $w$ and other variables $\xx$, then $T_\xx^\infty(w_0) \in \TT_\NMID(M + 1,N)$, where $T_\xx^\infty(w_0)$ is as in Proposition \ref{propositionTxalpha}.
\end{itemize}
Suppose that for all $k\in\N$ there exist arbitrarily large $k$-fold reflecting ordinals. Then for all $P,A,X,\xx$ there exists $\eta$ such that $\iota_\eta(P[\xx]) \neq U$, $j_\eta(A[\xx]) \neq U$, and $j_\eta'(X[\xx]) \neq U$.
\end{corollary}
\begin{proof}
%We need the following proposition: There exist $\eta_k$ increasing to $\eta = f_z(\tau) < \tau$ such that $f_{z_k}(\eta) = \eta_k$. Suppose $T$ is such that for all $w\in L_\eta \cap \P(\N)$, we have $j'_\eta(T[w]) \in \P(\N)$. Fix $w\in L_\eta\cap \P(\N)$. Then $w\in L_{\eta_k}$ for some $k$. It follows that $j'_\eta(T^\infty[w]) \in \P(\N)$.
Let $k$ be the depth of $P,A,X$, i.e. the maximum number of nestings of applications of rule (6b). Let $\eta$ be a $k$-fold reflecting ordinal large enough so that $\xx \in L_{\eta}^M \times \eta^N$. Then there exist $\eta = \eta_0 < \ldots < \eta_k$ such that $\eta_i = f_{\eta_i}(\eta_{i+1})$ for all $i = 0,\ldots,k - 1$. By induction on $i$, for all $j = 0,\ldots,k - i$ we have $\iota_{\eta_j}(P[\xx]), j_{\eta_j}(A[\xx]), j'_{\eta_j}(X[\xx]) \neq U$ for all $P,A,X$ of depth $\leq i$ and $\xx \in L_{\eta}^M \times \eta^N$. Letting $i = k $ and $j = 0$ yields the conclusion.
\end{proof}

\begin{remark*}
We call this language $\NMID$ rather than $\ID$ because the notation $\ID$ is standardly used to refer to the language of monotonic inductive definitions, i.e. the one generated by allowint terms of the form $T_\xx^\infty(w_0)$ for operators $T_\xx$ such that $w_1 \subseteq w_2$ implies $T_\xx(w_1) \subseteq T_\xx(w_2)$.
\end{remark*}

%\begin{definition}[Explain the ``theory of large ordinal axioms'' we are trying to generate\internal]
%Let $A_1$ and $A_2$ be two extensions of $\KP$. We write $A_1 < A_2$ if $A_2$ proves that there exists an ordinal $\alpha$ such that  $f(\alpha) = \alpha$ and $L_\alpha$ is a model of $A_1$.
%\end{definition}
%
%[Question: Does $\KP + \omega_2$ (or a more powerful theory) imply the existence of a model of $\KP + \omega_1$ less than $\theta_\smallemptyset$?\internal]
%
%[Note: $\Sigma_1$-separation holds in $\omega_1$ but it's not clear whether it holds in $\theta_\smallemptyset$\internal]

\section{Proof of Theorem \ref{theorem2}}
\label{sectionreverse}

In this section we prove a converse to Corollary \ref{corollaryID}:

%for all $T$, there exists an ordinal $|T|$ and a function $|T|\ni \alpha \mapsto T^\alpha$ satisfying the inductive definition. This is consistent with an entirely natural-numbers-based way of looking at things, because we can embed the theory of ordinals into the theory of natural numbers -- in particular, it is equivalent to writing $T^\infty(\PP)$, $\prec$ as in the other file\internal

\begin{theorem}
\label{theorem2}
Suppose that for all $P,A,X$ in $\NMID$, there exist arbitrarily large $\eta$ such that $j_\eta(A) \neq U$. Then for all $k\in\N$, there exist arbitrarily large $k$-fold reflecting ordinals.
%If $\alpha = \alpha_0$ is the smallest [or do we need a tower of smallness?\internal] ordinal such that there exist $\alpha_0 < \alpha_1 < \ldots < \alpha_k$ such that $\alpha_i = f_{\{\alpha_0,\ldots,\alpha_{i - 1}\}}(\alpha_{i + 1})$ for all $i = 0,\ldots,k - 1$, then the truth predicate of $L_\alpha$ is describable in $\NMID$, i.e. there is a term $T_k$ such that the set of true sentences of $L_\alpha$ are equal to $j(T_k)$.
\end{theorem}
\begin{proof}
Let $\AA$ be a countable set, and let $\FF_\GDST(\AA)$ denote the class of formulas of $\GDST$ with free ordinal variables of the form $\alpha_A$, where $A\in \AA$. We define a map $T_\AA:\P(\FF_\GDST(\AA)) \to \P(\FF_\GDST(\AA))$ as follows: If $\PP \subseteq \FF_\GDST(\AA)$, then
\begin{align*}
T_\AA(\PP) = \PP
&\cup \{\LQ X\in Y\RQ : \LQ X = Z\RQ, \; \LQ Z\in Y\RQ \in \PP\}\\
&\cup \{\LQ X \notin Y\RQ : \LQ \forall x\in Y \;\; x\neq X\RQ \in \PP\}\\
&\cup \{\LQ X = Y\RQ : \LQ \forall z\in X \;\; z\in Y\RQ, \LQ \forall z\in Y \;\; z\in X\RQ \in \PP\}\\
&\cup \{\LQ X \neq Y\RQ, \LQ Y \neq X\RQ : \LQ Z\in X\wedge Z\notin Y\RQ \in \PP\}\\
&\cup \{\LQ P\wedge Q\RQ : \LQ P\RQ, \LQ Q\RQ \in \PP\}\\
&\cup \{\LQ P\vee Q\RQ : \LQ P\RQ\in \PP \text{ or } \LQ Q\RQ\in\PP\}\\
&\cup \{\LQ \exists x\in X \;\; P(x)\RQ : \LQ Y\in X\RQ, \LQ P(Y)\RQ \in \PP\}\\
&\cup \{\LQ \{x\in X : P(x)\}\downarrow\RQ : \LQ P(Y)\RQ, \LQ \forall x\in X \;\; P(x)\downarrow\RQ \in \PP\}\\
&\cup \{\LQ Y\in \{x\in X : P(x)\}\RQ : \LQ Y\in X\RQ, \LQ P(Y)\RQ, \LQ \forall x\in X \;\; P(x)\downarrow\RQ \in \PP\}\\
%&\cup \{\LQ \forall x\in \{y\in X : P(y)\} \;\; Q(x)\RQ : \LQ \{y\in X : P(y)\}\downarrow\RQ\in\PP, \LQ \forall x\in X \;\; \neg P(x) \vee Q(x)\RQ\in\PP\}\\
%&\cup \{\LQ \forall x\in L_A \;\; P(x)\RQ : \LQ L_A\downarrow\RQ\in\PP, \;\;\forall X\in\TT_\GDST(\AA) \;\; \LQ X\in L_A\RQ\in \PP \Rightarrow \LQ P(X)\RQ\in\PP\}\\
&\cup \{\LQ L_A\downarrow\RQ : \forall Y \in \LL_{A,\PP} \;\; \LQ Y\downarrow\RQ \in\PP\}\\
&\cup \{ \LQ X \in L_A\RQ : X \in \LL_{A,\PP} , \;\; \forall Y \in \LL_{A,\PP} \;\; \LQ Y\downarrow\RQ \in\PP\}\\
\text{ where } \LL_{A,\PP} &= \{\LQ\{x\in L_B : P(x,Y_1,\ldots,Y_n,L_B)\}\RQ : \LQ B < A\RQ \in \PP, \;\; P\in \Delta_0(\ST), \LQ Y_j\in L_B\RQ \in \PP \;\;\forall j = 1,\ldots,n\}\\
&\cup \{\LQ \min\{\alpha : P(\alpha)\}\downarrow\RQ: \LQ P(A)\RQ, \LQ \forall \beta < A \;\;\neg P(\beta)\RQ \in \PP\}\\
&\cup \{\LQ B < \min\{\alpha : P(\alpha)\}\RQ: \LQ B < A\RQ, \; \LQ P(A)\RQ, \; \LQ \forall \beta < A \;\;\neg P(\beta)\RQ \in \PP\}\\
&\cup \{\LQ \forall x\in X \;\; P(x)\RQ : \LQ X\downarrow\RQ\in\PP, \;\;\forall Y\in\TT'_\GDST(\AA) \;\; \LQ Y\in X\RQ\in \PP \Rightarrow \LQ P(Y)\RQ\in\PP\}
%&\cup \{\LQ A <^* \alpha_i \RQ : A\in \TT_\GDST(0,i), \;\;\forall B \;\; \LQ B < A \RQ \in \PP \Rightarrow \LQ B <^* \alpha_i\RQ \in \PP\}
\end{align*}
%Here $\Delta_0(\ST)$ refers to the set of $\Delta_0$ formulas of $\ST$.

\begin{definition}
A set $\PP$ will be called \emph{good} if there exist an ordinal $q = q(\PP)$, a set $\AA = \AA_\PP$, and a function $g = g_\PP : \AA_\PP \twoheadrightarrow q(\PP)$ such that
\begin{equation}
\label{PPf}
\PP \subseteq \{\LQ \alpha_A\downarrow\RQ : A\in \AA\} \cup \{\LQ \alpha_B < \alpha_A\RQ : A,B\in \AA, g(B) < g(A)\}
\end{equation}
and
\begin{equation}
\label{frec}
\LQ \alpha_A\downarrow\RQ \in\PP,\;\; g(A) = \{g(B) : \LQ \alpha_B < \alpha_A\RQ \in \PP\} \text{ for all } A\in\AA.
\end{equation}
Note that if $\PP$ is good, then $q(\PP)$, $\AA_\PP$, and $g_\PP$ are all uniquely defined.
\end{definition}

\begin{lemma}
\label{lemmagood}
If $\PP$ is good, then $T_\AA^\infty(\PP) = \{P\in\FF_\GDST(\AA_\PP) : \iota_{q(\PP)}(P[g]) = \top\}$.
\end{lemma}
\begin{proof}
One checks by induction that for all ordinals $\gamma$,
\begin{itemize}
\item[(1)] $T_\AA^\gamma(\PP) \subseteq \{P\in\FF_\GDST(\AA) : \iota_{q(\PP)}(P[g]) = \top\}$ and
\item[(2)] $j'_{q(\PP)}(X[g]) = \{j'_{q(\PP)}(Y[g]) : \LQ Y \in X \RQ \in T_\AA^\gamma(\PP)\}$ for all $X\in \TT'_\GDST(\AA)$ such that $\LQ X\downarrow\RQ \in T_\AA^\gamma(\PP)$.
\end{itemize}
It follows that both of these relations hold with $\gamma$ replaced by $\infty$. The next step is to prove by set induction on $x = j'_{q(\PP)}(X[g])$ that for all $Y\in \TT'_\GDST(\AA)$ with $j'_{q(\PP)}(Y[g]) \in x$ (resp. $j'_{q(\PP)}(Y[g]) = x$), we have $\LQ Y \in X\RQ \in T_\AA^\infty(\PP)$ (resp. $\LQ Y = X\RQ \in T_\AA^\infty(\PP)$). Finally, one proves by induction on the depth of $P\in \FF_\GDST(N,0)$ and $X_1,\ldots,X_N \in \TT'_\GDST(\AA)$ that $\iota_{q(\PP)}(P(X_1[g],\ldots,X_N[g])) = \top$ implies $P(X_1,\ldots,X_N) \in T_\AA^\infty(\PP)$.
\end{proof}

Now let $T_{k + 1} = \id$, and for each $i = k,\ldots,0$ and good $\PP$ let
\begin{align*}
R_i(\PP) &= \PP \cup \{\LQ \alpha_{\eta_i}\downarrow\RQ\} \cup \{\LQ A < \alpha_{\eta_i}\RQ : \LQ A \downarrow\RQ \in \PP\}\\
S_i(\PP) &= T_{i + 1}^\infty(R_i(\PP))\\
S_i'(\PP) &= T_{\AA_{S_i(\PP)}}^\infty(S_i(\PP))\\
T_i(\PP) &= \PP\cup \{ \LQ \alpha_A\downarrow\RQ : \LQ A\downarrow\RQ \in S_i'(\PP) \cap \FF_\GDST(\AA_\PP), \; \LQ A = \alpha_i\RQ \in S_i'(\PP)\}\\
&\hspace{0.295in} \cup\{\LQ \alpha_B < \alpha_A\RQ : \LQ \alpha_B < A\RQ \in S_i'(\PP) \cap \FF_\GDST(\AA_\PP), \LQ A = \alpha_i\RQ \in S_i'(\PP)\}.
\end{align*}
%where $\pi(A)$ is the result of replacing all occurrences of terms of the form $\alpha_B$ in $A$ by corresponding terms $B$, yielding an element of $\TT_\GDST(\AA_\PP)$.

%\begin{claim}
%If $\PP$ is $i$-good with respect to $\bfalpha$, then so is $T_i(\PP)$.
%\end{claim}
%\begin{proof}
%We proceed by reverse induction. First we show that the claim is true for $i = k$. [Continue\internal]
%%
%If $\PP$ is $i$-good with respect to $\bfalpha$, then $\PP \cup \{\LQ \alpha_i\downarrow\RQ\} \cup \{\LQ A < \alpha_i\RQ : \LQ A \downarrow\RQ \in \PP\}$ is $(i + 1)$-good with respect to $(\bfalpha,\alpha_i)$ where $\alpha_i = g_\bfalpha(\AA)$. By our induction hypothesis, it follows that $S_i(\PP)$ is $(i + 1)$-good with respect to $(\bfalpha,\alpha_i)$. From the definition of $T_i$ it follows that $T_i(\PP)$ is $i$-good with respect to $\bfalpha$.
%\end{proof}

\begin{lemma}
\label{lemmaAi}
Let $\PP$ be a good set, and let
\begin{align*}
A_k(\alpha_{k - 1}) &= \alpha_{k - 1} + 1\\
A_i(\alpha_{i - 1}) &= \min\{ \alpha > \alpha_{i - 1} : F_\alpha(A_{i + 1}(\alpha)) = \alpha\}
\end{align*}
where $F_\alpha(\beta)$ is a term of $\GDST$ representing $f_\alpha(\beta)$. Then there exists $\eta$ large enough so that for all $j = i,\ldots,k$ we have $\alpha_j := j_\eta(A_j(\cdots(A_i[\alpha_{i - 1}]))) \neq U$. Moreover,
\[
q(T_i(\PP)) = \begin{cases}
q(\PP) + 1 & \text{ if $i \leq k$ and } q(\PP) < f_{q(\PP)}(j_\eta(A_i[q(\PP)]))\\
q(\PP) & \text{ otherwise}.
\end{cases}
\]
\end{lemma}
\begin{proof}
We inductively compute the $q$-value of the sets in question, while proving that the sets are good. Suppose $\PP$ is good. Then obviously $R_i(\PP)$ is good, $q(R_i(\PP)) = q(\PP) + 1$, and $g_{R_i(\PP)}(\LQ \eta_i\RQ) = q(\PP)$.

Next, % if $i = k$ then $S_k(\PP)$ is good and $q(S_k(\PP)) = q(\PP) + 1 = j_\eta(A_k[q(\PP)])$, where $\eta > q(\PP) + 1$.
 by the inductive hypothesis, for $i \leq k$ $S_i(\PP)$ is good and satisfies
\begin{align*}
q(S_i(\PP)) &= q(\PP) + 1 + \min\{\gamma : q(T_{i + 1}^\gamma(R_i(\PP))) = q(T_{i + 1}^{\gamma + 1}(R_i(\PP)))\}\\
&= q(\PP) + 1 + \min\{\gamma : i = k \text{ or } q(T_{i + 1}^\gamma(R_i(\PP))) = f_{q(T_{i + 1}^\gamma(R_i(\PP)))}(j_\eta(A_{i + 1}[q(T_{i + 1}^\gamma(R_i(\PP)))]))\}\\
&= j_\eta(\min\{\alpha \geq [q(\PP)] + 1 : i = k \text{ or } F_{\alpha}(A_{i + 1}(\alpha)) = \alpha\})\\
&= j_\eta(A_i[q(\PP)])
\end{align*}
for some $\eta$ sufficiently large, by hypothesis. Thus, for $i \leq k$, $T_i(\PP)$ is good and by Lemma \ref{lemmagood} satisfies
\begin{align*}
q(T_i(\PP)) &= \begin{cases}
q(\PP) + 1 & \text{ if there exists } A\in \TT_\GDST(\AA_\PP) \text{ such that } j_{q(S_i(\PP))}(A[g_\PP]) = q(\PP)\\
q(\PP) & \text{ otherwise}
\end{cases}\\
&= \begin{cases}
q(\PP) + 1 & \text{ if } f_{q(\PP)}(q(S_i(\PP))) \neq q(\PP)\\
q(\PP) & \text{ otherwise}
\end{cases}
\end{align*}
verifying the inductive hypothesis.
\end{proof}

To conclude the proof of Theorem \ref{theorem2}, we observe that for every ordinal $\alpha$, by Lemma \ref{lemmaAi} there exists $\eta$ such that $\beta := j_\eta(A_0[\alpha]) \neq U$. It is clear that $\beta > \alpha$ is a $k$-fold reflecting ordinal.
\end{proof}

\ignore{

\appendix
\section{Transfinite recursion}
\label{appendix1}

\begin{lemma}
\label{lemmarecursion}
Let $T(w,\xx)$ be a set term depending on set variables $w,\xx$. Then there exists a set term $(\alpha,\xx) \mapsto T_\xx^\alpha$ such that for all $z$, for all $\theta = f_z(\tau) < \tau$, and for all $\xx\in L_\theta^M$, the function $h_\xx : \theta \to L_\theta\cup \{U\}$ defined by
\[
h_\xx(\alpha) := j'_\theta(T_{[\xx]}^{[\alpha]})
\]
satisfies
\begin{equation}
\label{hxalpharecursive}
h_\xx\given\alpha \in L_\theta \text{ and } h_\xx(\alpha) = j'_\theta(T[h_\xx\given \alpha,\xx])
\text{ whenever } h_\xx(\beta) \neq U \;\;\forall \beta < \alpha.
\end{equation}
\end{lemma}
\begin{proof}
If $\eta < \theta$, we call a function $h:\eta \to L_\theta$ \emph{$(T,\theta)$-compatible} if
\[
h(\alpha) = j'_\theta(T[h\given \alpha,\xx]) \text{ for all } \alpha < \eta.
\]
It is not hard to show that there is a formula $P(h,\xx)$ such that $\iota_\theta(P[h,\xx]) = \top$ if and only if $h$ is $(T,\theta)$-compatible. We now define
\begin{align*}
H_{\xx,\eta} &:= \{(\alpha,y) : \exists h\in L_\eta \text{ such that } P(h,\xx) \text{ and } h(\alpha) = y\}\\
T_\xx^\alpha &:= H_{\xx,\min\{\eta : H_{\xx,\eta}(\alpha)\downarrow\}}(\alpha).
\end{align*}
(Here $H(\alpha)$ denotes function application, so that $j'_\theta(H[\alpha])$ is the image of $\alpha$ under the function $j'_\theta(H)$, rather than denoting that the $\GDST$ term $H$ has $\alpha$ as a free variable.)

We need to show that $T_\xx^\alpha(w_0)$ satisfies \eqref{hxalpharecursive}. Suppose $h_\xx(\beta) \neq U$ for all $\beta < \alpha$. It follows from the definition of $h_\xx$ that
\[
\iota_\tau(\forall \beta < [\alpha] \;\exists \eta < [\theta] \; H_{[w_0,\xx],\eta}(\beta)\downarrow ) = \top
\]
and by the definition of $f_z$ together with the fact that $\theta = f_z(\tau) < \tau$, there exists $\eta < \theta$ such that
\[
\iota_\tau(\forall \beta < [\alpha] \;\exists \eta' < [\eta] \; H_{[w_0,\xx],\eta'}(\beta)\downarrow ) = \top
\]
and thus
\[
\iota_\theta(\forall \beta < [\alpha] \; H_{[w_0,\xx,\eta]}(\beta)\downarrow )
= \iota_\tau(\forall \beta < [\alpha] \; H_{[w_0,\xx,\eta]}(\beta)\downarrow ) = \top
\]
from which it follows that $h_\xx\given \alpha = j'_\theta(H_{[w_0,\xx,\eta]}) \cap (\alpha\times L_\theta) \in L_\theta$. Now suppose that $y = j'_\theta(T[h_\xx\given\alpha,\xx]) \neq U$. The function
\[
h := ( h_\xx\given\alpha ) \cup \{(\alpha,y)\} \in L_\theta
\]
is $(T,\theta)$-compatible and thus $h \subseteq j'_\theta(H_{[\xx,\eta]}) \subseteq h_\xx$ for some $\eta < \theta$. Thus $h_\xx(\alpha) = y$, which completes the proof.
\end{proof}

\begin{corollary}
\label{corollaryGP}
If $P(\alpha)$ is a formula, then there exists an ordinal term $G_P(\alpha)$ such that for $z$ and $\theta = f_z(\tau) < \tau$, if $\gamma \leq \theta$ is the $\alpha$th ordinal such that $\iota_\theta(P[\gamma]) = \top$, then $\gamma = j_\tau(G_P[\alpha])$.
\end{corollary}
\begin{proof}
Indeed, let
\[
G_P(\alpha) = T(\alpha\ni \beta \mapsto T^\beta)
\]
where $T(h) = \min\{ \alpha : P(\alpha) \text{ and } \alpha\notin \pi_2(h)\}$.
\end{proof}

}% end ignore

\bibliographystyle{amsplain}

\bibliography{bibliography}

\end{document}